\documentclass[a4paper,reqno,final]{amsart}

\usepackage{color}

\usepackage[utf8]{inputenc}      
\usepackage{amsfonts}            
\usepackage{esint}               
\usepackage{cite}                
\usepackage{graphicx}
\usepackage{booktabs}            
\usepackage{titlesec}            
\usepackage{caption}
\titleformat{\section}{\normalfont\Large\bfseries}{\thesection}{1em}{}
\titleformat{\subsection}[runin]{\normalfont\bfseries}{\thesubsection.}{0.5em}{}[.]

\numberwithin{table}{section}    
\numberwithin{figure}{section}   
\numberwithin{equation}{section} 

\newcommand{\R}{\mathbb{R}}
\newcommand{\N}{\mathbb{N}}

\renewcommand{\AA}{\mathcal{A}}
\newcommand{\BB}{\mathcal{B}}

\newcommand{\DD}{\mathcal{D}}
\newcommand{\EE}{\mathcal{E}}

\newcommand{\GG}{\mathcal{G}}
\newcommand{\HH}{\mathcal{H}}

\newcommand{\JJ}{\mathcal{J}}

\newcommand{\LL}{\mathcal{L}}

\newcommand{\XX}{\mathcal{X}}
\newcommand{\YY}{\mathcal{Y}}

\newcommand{\dive}{\operatorname{div}}
\newcommand{\be}{\begin{equation}}
\newcommand{\ee}{\end{equation}}
\newcommand{\rot}{\text{rot}}

\newcommand{\beq}{\begin{equation}}
\newcommand{\eeq}{\end{equation}}

\newcommand{\curl}{\operatorname{curl}}

\newcommand{\pgtonew}{\sqrt{(1+\frac{\| p(t)\|^2}{(m_0^q c)^2})} }

\newcommand{\Hc}{\operatorname{H_{\curl}^{\Gamma}}}
\newcommand{\dist}{\operatorname{dist}}

\newcommand{\vektor}[2]{\begin{pmatrix} #1\\#2\end{pmatrix}}

\newcommand{\dual}[2]{\langle #1 , #2 \rangle}

\providecommand{\supp}{\operatorname{supp}}

\newcommand{\ddp}[2]{\frac{\partial #1}{\partial #2}}
\newcommand{\sdual}[2]{\left\langle #1 , #2 \right\rangle}

\newcommand{\embed}{\hookrightarrow}

\usepackage{mathrsfs}

\newtheorem{theorem}{Theorem}[section]
\newtheorem{lemma}[theorem]{Lemma}

\newtheorem{assumption}[theorem]{Assumption}
\newtheorem{remark}[theorem]{Remark}
\newtheorem{definition}[theorem]{Definition}

\begin{document}

\title[Optimal control of ML equations]
{Optimal control of the inhomogeneous relativistic Maxwell Newton Lorentz equations}

\author{C.~Meyer}
\address[C.~Meyer]{TU Dortmund, Faculty of Mathematics, Vogelpothsweg 87, 44227 Dortmund, Germany.}
\email{cmeyer@math.tu-dortmund.de}

\author{ S.~M.~Schnepp}
\address[S.~M.~Schnepp]{Institute of Geophysics, Department of Earth Sciences, ETH Zurich, Sonneggstrasse 5, CH-8092 Zurich, Switzerland.}
\email{schnepps@ethz.ch}

\author{O.~Thoma}
\address[O.~Thoma]{TU Dortmund, Faculty of Mathematics, Vogelpothsweg 87, 44227 Dortmund, Germany.}
\email{othoma@math.tu-dortmund.de}

\maketitle

\begin{abstract}
 This note is concerned with an optimal control problem governed by the relativistic Maxwell-Newton-Lorentz equations, 
 which describes the motion of charges particles in electro-magnetic fields and 
 consists of a hyperbolic PDE system coupled with a nonlinear ODE. An external magnetic field acts as control variable. 
 Additional control constraints are incorporated by introducing a scalar magnetic potential which leads to an additional 
 state equation in form of a very weak elliptic PDE.
 Existence and uniqueness for the state equation is shown and the existence of a global optimal control is established. 
 Moreover, first-order necessary optimality conditions in form of Karush-Kuhn-Tucker conditions are derived. 
 A numerical test illustrates the theoretical findings.
\end{abstract}

\textbf{Key words.} Optimal control, Maxwell's equation, Abraham model, Dirichlet control, state constraints.

\textbf{AMS subject classification.} 49J20, 49J15, 49K20, 49K15, 35Q61

\section{Introduction}

In this paper we discuss an optimal control problem governed by the relativistic Maxwell-Newton-Lorentz equations. 
This system of equations consists of Maxwell's equations, i.e., a hyperbolic PDE system, and a nonlinear ODE. 
It models the relativistic motion of charged particles in electromagnetic fields and is therefore used for the 
simulation of particle accelerators~\cite{Gjonaj2006,Li:dn,Schneidmiller:2007tx,Geddes:2004bi}.
The control variable is an additional exterior magnetic field, which, in practice, could be realized by exterior 
(dipole, quadrupole etc.) magnets surrounding the accelerator tube~\cite{Wille2000,Rossbach1992}. 
The aim of the optimization is to steer the particle beam to a given desired track and/or end-time position.
Beside the Maxwell-Newton-Lorentz system, the optimization problem is subject to several additional constraints. 
First, the particle beam should stay inside the accelerator tube, which is realized by pointwise constraints on the particle 
position and constitutes a pointwise state constraint from a mathematical point of view. 
Moreover, as a stationary magnetic field, the control has to satisfy certain constraints, e.g.\ its divergence has to vanish. 
In order to guarantee these constraints, we introduce a scalar magnetic potential, whose boundary data serve as new control 
variable. This gives rise to a Poisson equation for the exterior magnetic field entering the system of state equations. 
Physically, the new control variable can be interpreted as a surface current on the boundary of the computational domain. 
In this way we obtain a Dirichlet boundary control problem.

Let us put our work into perspective. Optimal control of Maxwell's equations and coupled systems involving 
these have been subject to intensive research in the recent past. We only mention the work of Tr\"oltzsch et al. \cite{DruetKleinSprekelsYouseptTroeltzschYousept2011, YouseptTroeltzsch2012, NicaiseStingelinTroeltzsch2014, NicaiseTroeltzsch2014, NicaiseStingelinTroeltzsch2015} and Yousept \cite{Yousept2010, Yousept2012:1, Yousept2012:2,Yousept2013:1,Yousept2013:2}.
However, most of these contributions deal with stationary or time harmonic Maxwell's equation. 
In  \cite{NicaiseTroeltzsch2014} the so-called evolution Maxwell equation in form of a (degenerate) parabolic PDE is considered. 
In contrast to this, we deal with a first-order hyperbolic system for the electric and the magnetic fields.
Optimal control of magneto-hydrodynamic processes was investigated in \cite{GriesseKunisch2006}. 
These processes are modeled by a coupled system consisting of Maxwell's equation and the Navier-Stokes equations. 
However, \cite{GriesseKunisch2006} also focuses on the stationary case.
Up to our best knowledge, the non-standard coupling of the
(hyperbolic) Maxwell's equation and the ODE for the relativistic motion of charged 
particles have not been treated so far in the context of optimal control, neither from an analytical nor from a numerical point of view. 
The mathematical treatment of the Maxwell-Newton-Lorentz system itself however has been investigated by several authors before. 
Concerning the analysis we mention \cite{Spohn2004,ImaikinKomechMauser2004, BauerDeckertDuerr2013, Falconi2014}
and the references therein. Regarding its numerical treatment we refer to~\cite{Gjonaj2006,Li:dn,Geddes:2004bi}.
The analytical and numerical investigations presented in this paper will partly rely on these findings.
As mentioned before the control constraints on the external magnetic field are realized by introducing a scalar potential 
which leads to a boundary control problem of Dirichlet type. Optimal control problems of this type have been 
intensely investigated in the recent past, see e.g.\ \cite{CasasRay2006, KunischVexler2007, DeckelnickGuntherHinze2009, OFPhanSteinbach2010, MayRannacherVexler2013}. 
We choose $L^2(\Gamma)$ as control space, so that the associated Poisson equation is treated in very weak form, 
which is a well-established procedure, cf.\ e.g.\ \cite{MayRannacherVexler2013}.
Another challenging aspect of the optimal control under consideration are the pointwise state constraints on the particle position. 
Lagrange-multipliers associated with constraints of this type, in general, lack in regularity and are only measures, 
see e.g.\ \cite{Casas1986, Casas1993} for the case of PDEs and \cite{HartlSethiVickson1995} and the references therein for the case of ODEs.
Numerically, such constraints are frequently treated by regularization and relaxation methods, especially in the PDE case, 
cf.\ e.g.\ \cite{HK06, MRT06, Schiela2013}. We also follow this approach and apply an interior point method 
to realize the state constraints.

The paper is organized as follows: in the following section we introduce the physical model, i.e., the Maxwell-Newton-Lorentz system. 
This model is not directly amenable for a mathematically rigorous treatment mainly due to two reasons, which are addressed at the end of 
Section \ref{sec:statphmod}. We therefore slightly modify the model in Section \ref{sec:optcon} by replacing the point charge with a 
distributed volume charge density. In addition the scalar magnetic potential is introduced in this section which allows us to formulate the 
optimal control problem, first in a formal way. After stating our standing assumptions in Section \ref{sec:assu}, Section \ref{sec:rigorous} 
is then devoted to a mathematically sound and rigorous statement of the optimal control problem, including the function spaces for all 
optimization variables as well as the notion of solutions for the differential equations involved in the state system.
We start the analysis of the optimal control problem by discussing the state equation in Section \ref{sec:state}. 
Then we turn to the optimal control problem and show the existence of globally optimal controls in Section \ref{sec:exopt}. 
The analytical part of the paper ends with the derivation of first-order-necessary optimality conditions involving Lagrange multipliers 
in Section \ref{sec:nec}. 
The final Section \ref{sec:diskret} is devoted to the numerical treatment of the optimal control problem. 
After describing the discretization of the state system and the optimization algorithm, we present an exemplary 
numerical result for the end time tracking of a single-particle beam.

\section{Statement of the physical model}\label{sec:statphmod}

In this section we introduce the physical model underlying the optimal control problem. 
The precise mathematical model will be stated in Section \ref{sec:rigorous}.

To keep the discussion concise we will restrict to the motion of only one particle in the accelerator. 
The adaptation of the model to a finite number of particles is straightforward, see Remark \ref{rem:manyparts} below.
Our model is based on the classical inhomogeneous Maxwell's equations with the boundary conditions of a perfect conductor. 
In strong form these equations read:
\begin{subequations}\label{eq:maxwell}
	\begin{align} 
		  &\epsilon \frac{\partial}{\partial t} E(x,t) - \mu^{-1} \curl \ B(x,t) =  j(x,t)&  &\text{in } \Omega \times [0,T] \label{eq:maxwella}\\
		  &\frac{\partial}{\partial t} B(x,t) + \curl E(x,t) = 0  & &\text{in } \Omega \times [0, T] \label{eq:maxwellb}\\
	          & \dive E(x,t) = \frac{1}{\epsilon} \rho(x,t), \quad \dive B(x,t) = 0&  &\text{in } \Omega \times [0,T]\label{eq:maxwellc}\\
	           &E(x,0)=E_0(x), \quad B(x,0)=B_0(x) \label{eq:maxwelld}&  &\text{in } \Omega \\
	            &  E \times n = 0, \quad B \cdot n = 0&  &\text{on } \Gamma \times [0,T].\label{eq:maxwelle}
 	\end{align}
\end{subequations}
Herein, $E$ and $B$ denote the electric and magnetic field, respectively, and $\Omega$ is the domain occupied by the interior of the 
accelerator channel. Its boundary $\partial\Omega$ is denoted by $\Gamma$, and $n$ is the outward unit normal on $\Gamma$.
Moreover, $\epsilon$ is the permittivity of free space, while $\mu$ denotes the permeability, which are assumed to be constant 
in $\Omega$.
Finally, $\rho$ and $j$ denote the charge density and the electric current.

\begin{remark}\label{rem:gauss}
 Provided the conservation of charge holds, 
 the two Gauss laws in \eqref{eq:maxwellc} as well as the boundary condition on $B$
 intrinsically follow from Faraday's and Ampere's laws in \eqref{eq:maxwella} and 
 \eqref{eq:maxwellb} so that \eqref{eq:maxwell} is not overdetermined. 
\end{remark}

In our case, the charge density is generated by a single point charge and therefore given by
\begin{equation}\label{eq:charge}
	\rho(x,t) := q \delta(|x-r(t)|_2) \quad \text{in } \Omega \times [0,T],
\end{equation}
where $q>0$ is the constant particle charge, $r$ denotes the particle position, and $|\,.\,|_2$ is the Euclidean norm of a vector. 
Furthermore, $\delta: \R \to \{0,\infty\}$ is the Dirac delta distribution.
The current $j(x,t)$ arising on the right hand side in \eqref{eq:maxwella} is generated by the motion of the particle
and thus given by
\begin{equation}\label{eq:currentcoupl}
 j(x,t) := -q\delta(|x-r(t)|_2)v(p(t)) \quad \text{in } \Omega \times [0,T],
\end{equation}
where $p$ denotes the relativistic momentum of the particle. Moreover, we set 
\begin{equation}\label{eq:sigma}
 v(p(t)) := (m_0^q \, \gamma(p(t)))^{-1}p(t)
\end{equation}
with the mass at rest $m_0^q$ and the Lorentz factor
\begin{equation*}
 \gamma(p(t)) := \pgtonew,
\end{equation*}
where $c>0$ denotes the speed of light. 
Note that $v(p)$ is nothing else than the velocity of the particle.
It is easily verified that $\rho$ and $j$ chosen in this way satisfy the conservation of charge.

We summarize the constants of the model in Table \ref{tabconstants}.
\begin{table}[h]
\begin{tabular}{lr}
\toprule
Physical constants& Name of quantity\\
\midrule
$c$& speed of light\\
$\epsilon$& permittivity\\
$\mu$ & permeability\\
$m_0^q$& rest mass\\
$q$ & particle charge\\
\bottomrule
\end{tabular}
\addtolength{\abovecaptionskip}{-22pt}
\caption{Overview of arising constants}\label{tabconstants}
\end{table}

In addition to \eqref{eq:sigma} we introduce the abbreviation
\begin{equation}\label{eq:sigmac}
 \beta(p(t)) :=c^{-1}v(p(t)),
\end{equation} 
which prove helpful in the sequel.

The motion of the particle in electromagnetic fields is governed by the relativistic Newton-Lorentz equations given by the formulae 
\begin{subequations}\label{eq:nlorentz}
	\begin{align}
	         & \dot{p}(t) = q\Big[e(r(t))+ E(r(t),t) +\beta(p(t)) \times \Big( b(r(t)) + B(r(t),t) \Big) \Big]   & & \text{in } [0,T] \label{eq:nlorentza}\\
        		&\dot{r}(t) =  v(p(t))  & &\text{in } [0,T]\label{eq:nlorentzb}\\
		& p(0)= p_0 \quad \text{and} \quad r(0) = r_0\label{eq:nlorentzc}
	\end{align}
\end{subequations}
with initial particle position and momentum $p_0, r_0\in \R^3$.
Furthermore, $e$ and $b$ denote the external electric and magnetic fields, respectively. 
These fields are generated by exterior capacitors and magnets in order to steer the particle beam. 
They are assumed to fulfill the homogeneous Maxwell's equations in $\Omega$. 
As we only consider magnets for manipulating the beam, we assume $e$ to equal zero.
Therefore, the external magnetic field $b$ has to satisfy the conditions
\begin{equation}\label{eq:controlcon}
	 \dive\, b = 0 , \ \curl \, b= 0 \ \text{ and } \ \partial_t b = 0 \quad \text{in} \ \Omega .
\end{equation}
This external magnetic field $b$ will serve as control in the following.  

To summarize the overall model reads as follows:
\begin{subequations} \label{eq:physmod}
\begin{align}
 &\epsilon \frac{\partial}{\partial t} E (x,t) -\mu^{-1} \curl B(x,t) =   
 - q\delta(|x-r(t)|_2)v(p(t)) &  &\text{in } \Omega \times [0,T] \label{eq:physmod_a}\\
 &\frac{\partial}{\partial t} B(x,t)+ \curl E(x,t)  = 0 &  &\text{in } \Omega \times [0,T] \\  
 &\dive E(x,t)  = \frac{1}{\epsilon} q\delta(|x-r(t)|_2), \quad  \dive B(x,t) =0 &  &\text{in } \Omega \times [0,T] \label{eq:physmod_c}\\ 
 &\dot p(t)  =q\Big( E(r(t),t)+\beta(p(t))\times \big(b(r(t))+B(r(t),t)\big)\Big)  & & \text{in } [0,T] \label{eq:physmod_d}\\
 &\dot r(t) =  v(p(t))& & \text{in } [0,T]\label{eq:physmod_e}\\
 &E(x,0)=E_0(x), \, B(x,0)=B_0(x), \, r(0)=r_{0}, \, p(0)=p_{0}, & &\text{in}\ \Omega\\
 & E \times n = 0, \quad B \cdot n = 0&  &\text{on } \Gamma \times [0,T].
\end{align}
\end{subequations}
 
\begin{remark}\label{rem:manyparts}
 In case of an entire bunch of $n$ particles the electric current is given by $- \sum_{i=1}^{n}q_i\delta(|x-r_i(t)|_2)v(p_i(t))$, 
 while the charge density becomes $\sum_{i=1}^{n} q_i\delta(|x-r_i(t)|_2)$. 
 The rest of the system remains unchanged, except that we {had} $n$ equations of the form \eqref{eq:physmod_d}, \eqref{eq:physmod_e} 
 for each of the $n$ particles, cf.\ e.g.\ \cite[Section 11]{Spohn2004}. It is therefore straightforward to adapt the analysis presented in the following to 
 the situation of $n$ particles.
\end{remark}

The model equations in \eqref{eq:physmod} feature two critical aspects. 
First, the particle must not leave the computational domain $\Omega$, i.e.\ the interior of the accelerator, since 
otherwise the right hand side in \eqref{eq:physmod_d} is not well defined. 
This issue will be resolved by adding an additional state constraints to the optimal control problem. 
From an application driven point of view this constraint is meaningful, too.
Secondly, the pointwise evaluation of the electric and the magnetic fields precisely at the point $x=r(t)$ in \eqref{eq:physmod_d} 
is, in general, not well defined, since solutions of Maxwell's equations with $j$ given by \eqref{eq:currentcoupl} are singular at this point.
We will overcome {this} difficulty by introducing the so-called Abraham model, which is addressed in the next section.
For further details on the Abraham model, we refer to \cite[Section 2.4]{Spohn2004}.

\section{The optimal control problem}\label{sec:optcon}

This section is devoted to the optimal control problem. Having established the Abraham model, 
we introduce a scalar potential to cope with the additional conditions in the external magnetic field in \eqref{eq:controlcon}.
Then, we state the complete optimal control problem including the objective functional and the additional state constraints 
on the particle position.
The rest of this section is concerned with the standing assumptions and the mathematically rigorous statement of the optimal control problem.

As described above, the pointwise evaluation in \eqref{eq:physmod_d} is, in general, not well defined. 
To resolve this issue, we replace the Dirac delta distribution by a smeared out version. For this purpose 
we fix a function $\varphi:\R^3 \to \R$ such that 
\begin{equation}\label{eq:smearout}
\begin{gathered}
 \varphi \in C^{2,1}(\R^3), \quad \supp(\varphi)\subseteq B_R(0), \quad \varphi(x) \geq 0\quad \forall\,x\in \R^3\\
 \int_{\R^3} \varphi(x)\,dx = 1, \quad \varphi(x) = \varphi(y) \quad \text{if } |x|_2 = |y|_2
\end{gathered}
\end{equation}
(i.e., $\varphi$ is rotationally symmetric).
The pointwise evaluations in \eqref{eq:physmod_d} are then approximated by
\begin{align}
 & E(r(t),t)+\beta(p(t))\times \big(b(r(t))+B(r(t),t)\big)\nonumber\\ 
 &\quad \approx \int_{\Omega} \varphi(x-r(t))\Big[ E(x,t)+  \beta(p(t))\times  \big(b(x) + B(x,t)\big)  \Big] dx.
 \label{eq:rhslorentz}
\end{align}
Accordingly, the charge distribution and the current density are replaced by 
\begin{equation}\label{eq:chargecurrent}
 \rho(x,t) = q\,\varphi(x-r(t)) \quad \text{and}\quad j(x,t) =  - q\,\varphi(x-r(t))v(p(t)).
\end{equation}
One readily verifies that the conservation of charge is also fulfilled by this choice for $\rho$ and $j$.

To incorporate the conditions on the external magnetic field in \eqref{eq:controlcon}, 
we introduce a scalar magnetic potential as solution of the following Poisson's equation with Dirichlet boundary data
\begin{equation}\label{eq:poisson}
 -\Delta \eta = 0 \quad \text{in} \ \Omega,\quad \eta = u \quad  \text{on} \ \Gamma.
\end{equation}
Under the assumption that $\Omega$ is a simply connected domain, the gradient $b:= \nabla \eta$ is a conservative vector field so that
\begin{equation*}
	 	\dive\, b = \dive \, \big( \nabla \eta \big) = \Delta \eta= 0, \quad
		\curl \, b= \curl \big(\nabla \eta \big)= 0, \quad
		\partial_t b = 0,
\end{equation*}
i.e.\ \eqref{eq:controlcon}, is fulfilled almost everywhere. 
The Dirichlet data $u$ in \eqref{eq:poisson} will serve as the new control variable in the following.
Employing \eqref{eq:poisson} and integration by parts, one rewrites the integral involving $b$ in \eqref{eq:rhslorentz} by
\begin{equation}\label{eq:intbypartsode}
\begin{aligned}
 \int_{\Omega} \varphi(x-r(t))\, \beta(p(t)) \times  b(x) dx
 &= - q\int_{\Omega}{\eta \,\nabla \varphi(x-r(t)) \times  \beta(p(t)) \ dx}\\
 &\quad + q\int_{\Gamma}{u\, \varphi(x-r(t))\, \beta(p(t)) \times n\  ds}.
\end{aligned}
\end{equation}
Summing up all components of the physical model, the optimal control problem under consideration reads
\begin{equation}\tag{$\tilde{\textup{P}}$}\label{eq:tildeP}
 \text{minimize} \quad \JJ(r,u):= \int_0^T J_1(r(t))\,dt + J_2(r(T)) + \frac{\alpha}{2} \int_\Gamma u^2\,d\varsigma
\end{equation}
subject to Maxwell's equations
\begin{subequations}\label{eq:maxwellstrong}
\begin{align}
 &\epsilon \,\frac{\partial}{\partial t} E (x,t) -\mu^{-1} \curl B(x,t) =   - q\varphi(x-r(t))v(p(t)) \quad \text{in } \Omega \times [0,T] \\
 &\frac{\partial}{\partial t} B(x,t)+ \curl E(x,t)  = 0 \quad \text{in } \Omega \times [0,T] \\  
 &\dive \ E(x,t)  = \frac{1}{\epsilon} q\varphi(x-r(t)), \quad  \dive \ B(x,t) =0 \quad \text{in } \Omega \times [0,T]\\ 
 &E(x,0)=E_0(x), \, B(x,0)=B_0(x) \quad \text{in}\ \Omega\\
 & E \times n = 0, \quad B \cdot n = 0 \quad \text{on } \Gamma \times [0,T],
\end{align}
\end{subequations}
the relativistic Newton-Lorentz equations
\begin{subequations}\label{eq:nlorentzpInt}
\begin{align}
 \dot{p}(t) &=
 q \int_{\Omega}{\varphi(x-r(t))\Big[ E(x,t)+  \beta(p(t))\times  B(x,t)  \Big] dx} \nonumber\\
 &\quad - q\int_{\Omega}{\eta\, \nabla \varphi(x-r(t)) \times  \beta(p(t)) \ dx} \label{eq:newtonlorentz}\\
 &\quad + q\int_{\Gamma}{u \,\varphi(x-r(t)) \,\beta(p(t)) \times n\  ds}  \qquad \text{in } [0,T] \nonumber\\ 
 \dot r(t) &= v(p(t)) \quad \text{in } [0,T]\\
 r(0)&=r_{0}, \quad p(0)=p_{0},
\end{align}
\end{subequations}
Poisson's equation
\begin{equation}\label{eq:poissonstrong}
 -\Delta \eta = 0 \quad \text{in} \ \Omega,\quad \eta = u \quad  \text{on} \ \Gamma,
\end{equation}
and pointwise state constraints on the particle position
\begin{equation}\label{eq:stateconst}
 r(t) \in \tilde\Omega.
\end{equation}
Herein, $J_1, J_2: \R^3 \to \R$ are given functions which reflect the goal of the optimization to steer the beam 
on the overall time interval and at end time, respectively. 
Moreover, the Tikhonov parameter $\alpha$ is a positive real number. 
Finally, $\tilde\Omega \subset \Omega$ is a closed subdomain fulfilling
\begin{equation*}
 \dist(\tilde{\Omega},\Gamma) > R,
\end{equation*}
where $R$ is the number defining the support of the smeared out delta distribution, cf.\ \eqref{eq:smearout}.

\begin{remark}\label{rem:stateconst}
 Note that now the integrands on the right-hand side of \eqref{eq:newtonlorentz} are well-defined in any case, 
 even if $r(t) \notin \Omega$ for some $t\in [0,T]$. 
 However, in this case, the model becomes physically meaningless.
 In this way the state constraint in \eqref{eq:stateconst} ensures that the model does not loose its physical validity. 
 Moreover, in applications, it is important to keep the particles inside the accelerator tube, which is also reflected by the condition \eqref{eq:stateconst}.
\end{remark}

\subsection{Standing assumptions and notation}\label{sec:assu}

We start by introducing several function spaces which will be useful in the sequel.

\begin{definition}[$H(\curl;\Omega)$-spaces]
 By $X$ we denote the space $X= L^2(\Omega;\R^3)$. 
 For convenience of notation the scalar products and corresponding norms in $X$ and $X\times X$ are both denoted by $(.,.)_X$ 
 and $\|.\|_X$, respectively. Moreover, we set 
 \begin{equation*}
  H(\curl;\Omega) := \{\omega\in X: \curl\omega \in X\},
 \end{equation*}
 where $\curl: X \to \DD'$ denotes the distributional curl-operator. 
 With the obvious scalar product $H(\curl;\Omega)$ becomes a Hilbert space. It is well known that 
 there exists a linear and continuous operator $\tau_n: H(\curl;\Omega) \to H^{-1/2}(\Gamma;\R^3)$ such that 
 $\tau_n \omega = \omega \times n$ for all $\omega \in H(\curl;\Omega) \cap C(\bar\Omega;\R^3)$, 
 see e.g.\ \cite[Chapter 2]{GiraultRaviart1986}.
 In the sequel we will denote $\tau_n\omega$ by $\omega \times n$ for all $\omega \in H(\curl;\Omega)$ for simplicity 
 and call this operator tangential trace.
 For a detailed discussion of the tangential trace we refer to \cite{Alonso1996}.
 Furthermore, we define the set
 \begin{equation*}
  \Hc := \left\{ V = (V_1, V_2) \in H(\curl;\Omega) \times H(\curl;\Omega): V_1 \times n=0\right\}.
 \end{equation*}
 As a closed subspace of a Hilbert space, $\Hc$ is a Hilbert space itself.
\end{definition}

\begin{definition} [$H(\dive;\Omega)$-spaces]
We define the set
\begin{equation*}
 H(\dive ;\Omega) := \left\{\omega \in X: \ \dive \omega \in L^2(\Omega)  \right\},
\end{equation*}
where $\dive: X \to \DD'$ denotes the distributional divergence. Equipped with the obvious scalar product,
$H(\dive ;\Omega)$ becomes a Hilbert space.
Functions in $H(\dive ;\Omega)$ admit a normal trace, i.e., there is a linear and continuous operator 
$\gamma_n: H(\dive ;\Omega) \to H^{-1/2}(\Gamma)$ such that $\gamma_n \omega = \omega\cdot n$ for all 
$\omega \in H(\dive ;\Omega) \cap C(\bar\Omega;\R^3)$, see e.g.\ \cite[Theorem 1.2]{Temam1977}. 
As above, we denote the normal trace by $\omega\cdot n$ for all $\omega$ in $H(\dive ;\Omega)$.
Furthermore, we define the set 
\begin{equation*}
 \HH :=\left\{ v \in H^1_0(\Omega): \ \nabla v \in H (\dive \ ;\Omega) ,\;\partial_n v \in L^2(\Gamma)\right\},
\end{equation*}
where we set $\partial_n v := n\cdot \nabla v$. Endowed with the norm 
$$\|v\|_{\HH}= (\big\|v\|_{H^1(\Omega)}^2 + \| \Delta v\|_{L^2(\Omega)}^2 
+ \| \partial_n v \|_{L^2(\Gamma)}^2\big)^{\frac{1}{2}}$$
and the corresponding scalar product, it is a Hilbert space, too. 
Here and in the following, $\Delta := \dive\nabla : \HH \to L^2(\Omega)$ denotes the Laplacian.
\end{definition}

Now we are in the position to state the assumptions on the domain $\Omega$.

\begin{assumption}[Regularity of the domain]\label{assu:domain}\ 
\begin{enumerate}
 \item The domain $\Omega \subset \R^3$ is open, bounded, and simply connected. 
 \item The subdomain $\tilde\Omega$ can be represented by
	\begin{equation*}
		\tilde\Omega = \left\{ x \in\R^3 \ : \ g_i(x) \leq 0, \ i= 1,...,m \right\} 
	\end{equation*}
	where $m \in \N$ and $g_i \in C^1(\R^3,\R)$ with absolutely continuous derivatives $g_i'$.
 \item Furthermore, $\Omega$ is such that for all $g \in L^2(\Omega)$ there exists a unique solution $w \in \HH$ of 
 \begin{equation}\label{eq:poissonweak}
  \int_{\Omega}  \nabla w \cdot \nabla v \, dx =  \int_{\Omega} g \, v \, dx	\quad \forall \, v \in  H^1_0(\Omega)
 \end{equation}
 and the following a priori estimate 
 $$\|w\|_{\HH} \leq C\, \| g\|_{L^2(\Omega)}$$
 is fulfilled with a constant $C>0$ independent of $g$ and $w$.
\end{enumerate}
\end{assumption}

\begin{remark}
 By the Lax-Milgram Lemma \eqref{eq:poissonweak} admits a unique solution in $w\in H^1_0(\Omega)$ and, due to $g\in L^2(\Omega)$, 
 it immediately follows that $\nabla w \in H (\dive \ ;\Omega)$. The additional condition $\partial_n w \in L^2(\Gamma)$ 
 is satisfied under rather mild assumptions on the boundary of $\Omega$, cf.\ \cite[Chapter 6]{Dauge1988}.
\end{remark}

\begin{assumption}[Problem data]\label{assu:further}
 We {assume} the following assumptions on the data in \eqref{eq:optconexact}:
 \begin{itemize}
  \item $r_0\in \tilde\Omega$.
  \item The first two contributions to the objective fulfill $J_1, J_2 \in C^1(\R^3)$. 
   Furthermore, we assume that $J_1$ and $J_2$ are bounded from below by constants 
  $\underline{c_1}> -\infty$ and $\underline{c_2}> -\infty$.
  \item The Tikhonov regularization parameter satisfies $\alpha\in \R$, $\alpha > 0$.
  \item The smeared out delta distribution $\varphi$ fulfills the assumptions in \eqref{eq:smearout}.
  \item $\epsilon$, $\mu$, $q$ are positive constants.
  \item $E_0, B_0 \in X$.
  \item $g_1, ..., g_m \in C^1(\R^3)$.
 \end{itemize}
\end{assumption}

Given a linear normed space $\XX$ we denote by $C_{\{0\}}([0,T];\XX)$ the space of functions from $C([0,T];\XX)$ which vanish at 
$t=0$. The space $C_{\{0\}}^1([0,T];\XX)$ is defined analogously. By
\begin{equation*}
 	Y := \{ (r,p) \in H^1(]0,T[;\R^3)^2 : r(0) = p(0) = 0\}, \quad
 	Z := L^2(]0,T[;\R^3)^2
\end{equation*}
we denote the state space, which comes into play in Section \ref{sec:nec}.
To keep the notation concise, we also denote the space $\{ r\in H^1(]0,T[;\R^3) : r(0) = 0 \}$ by $Y$.
In addition, the Jacobian of the electric current $j$ as given in \eqref{eq:chargecurrent} is denoted by
\begin{equation}\label{eq:jprime}
 j'(r,p) := \big(\partial_r j(r,p), \partial_p j (r,p)\big)
 =  
 \begin{pmatrix}
 	\partial_r j_1(r,p) & \partial_p j_1(r,p)\\
 	\partial_r j_2(r,p) & \partial_p j_2(r,p).  
 \end{pmatrix}
\end{equation}
If $\XX$ and $\YY$ are linear normed spaces, we write $\LL(\XX,\YY)$ for the space of linear and bounded operators from $\XX$ to $\YY$.
Furthermore, $|v|_2$ is the Euclidean norm of a vector $v\in \R^3$. Abusing the notation slightly, we 
denote the Euclidean norm on $\R^3\times \R^3$ by the same symbol, i.e., 
$|(v,w)|_2 := \sqrt{|v|_2^2 + |w|_2^2}$ for $v,w\in \R^3$.
If $A\in \R^{3\times 3}$, then $|A|_F$ denotes the Frobenius norm of $A$. Finally, throughout the paper, $C$ is a generic constant. 

\subsection{Mathematically rigorous formulation of the optimal control problem}\label{sec:rigorous}

In the following we define a rigorous notion of solutions to the system of state equations in
\eqref{eq:maxwellstrong}--\eqref{eq:poissonstrong}. We start with Maxwell's equation and define 
the linear and unbounded operator
\begin{equation*}
	\mathcal A: X \times X \rightarrow X \times X, \quad 
	\mathcal A := \left(\begin{matrix}
	0 & -\curl\\
	\curl & 0\\
    \end{matrix} \right)
\end{equation*}
with its domain of definition $D(\AA) = \Hc$. In view of Remark \ref{rem:gauss}, Maxwell's equation 
can then be reformulated by the following \emph{Cauchy-Problem}:
  	\begin{equation}\label{eq:cauchymax}
 		 \left.
  		 \begin{aligned}
    		 	\frac{\partial}{\partial t}\vektor{E(t)}{B(t)} + \AA \vektor{E(t)}{B(t)} &= j \quad  \text{a.e.\ in} \ [0,T] \\
     		\vektor{E(0)}{B(0)} &= \vektor{E_0}{B_0}\\
   		\end{aligned}
   		 \quad \right\}
  	\end{equation}
As shown in \cite[Chapter XVII.B., Section 4]{Dautray1992:1} and \cite[Chapter IX, Section 3]{Dautray1992:2}, 
$-i\AA$ is self-adjoint, i.e., $-i \mathcal A= i \mathcal A^*= -(i \mathcal A)^*$, 
and consequently the theorem of Stone {states} that $\AA$ is the infinitesimal generator of 
a $C_0$-semigroup, see \cite{Pazy1983}. We denote this semigroup and its two components by
\begin{equation}\label{eq:GGdef}
 \GG(t): X\times X \to X\times X,\quad \GG(t) := \vektor{\EE(t)}{\BB(t)}.
\end{equation}
As $\GG$ is strongly continuous, the following notion of solutions to \eqref{eq:cauchymax} is meaningful:

\begin{definition}[Mild solution of Maxwell's equations]\label{def:mildsol}
 Let $(E_0, B_0)  \in X \times X$ and $j \in L^1([0,T];X)^2$ be given. Then we call $(E, B) \in C([0,T]; X)^2$, 
 given by
 \begin{equation}\label{eq:mildsol}
  \vektor{E(t)}{B(t)} = \GG(t) \vektor{E_0}{B_0}+ \int_0^t{\GG(t-\tau)j(r,p)(\tau) \ d \tau} \quad 0 \leq t \leq T,
 \end{equation}
 \emph{mild solution}  of the Cauchy problem \eqref{eq:cauchymax} on $[0,T]$.
\end{definition}

Note that the strong continuity of $\GG$ implies that 
the right-hand side in \eqref{eq:mildsol} indeed defines an element of $C([0,T]; X)^2$.
Moreover, by strong continuity, there are constants $M\geq 1$ and $\omega \geq 0$ such that 
\begin{equation}\label{eq:semiexp}
 \|\GG(t)\|_{\LL(X\times X,X\times X)} \leq M e^{\omega\,t}\quad \forall\, t\in [0,T]
\end{equation}
giving in turn the following a priori estimate
\begin{equation}\label{eq:apriories}
 \|(E,B)\|_{C([0,T];X \times X)} \leq 2\,Me^{\omega\,T} \big( || (E_0,B_0)\|_{X \times X} + \| j \|_{L^1([0,T]; X \times X)} \big).
\end{equation}

Next we turn to the Poisson equation \eqref{eq:poissonstrong}. As the Dirichlet data are given by the control function $u\in L^2(\Gamma)$, 
we employ the following notion of solutions:

\begin{definition}[Very weak solution of Poisson equation]\label{def:vweak}
 For given $u \in L^2(\Gamma)$ we call $\eta \in L^2(\Omega)$ \emph{very weak solution} of \eqref{eq:poissonstrong}, 
 if it solves the very weak formulation 
 \begin{equation}\label{eq:vweakpoisson}
  -\int_{\Omega}{\eta \Delta v \ dx} + \int_{\Gamma}{u \, \partial_n v \ d \varsigma} = 0 \quad \forall \,  v \in \mathcal H.
 \end{equation}
\end{definition}

\begin{lemma}\label{le:existencePos}
 For every $u\in L^2(\Gamma)$ there exists a unique solution $\eta\in L^2(\Omega)$ of \eqref{eq:vweakpoisson} satisfying an 
 a priori estimate
 \begin{equation*}
  \|\eta\|_{L^2(\Omega)} \leq C\, \|u\|_{L^2(\Gamma)}
 \end{equation*}
 with a constant $C>0$ independent of $u$ and $\eta$.
\end{lemma}

\begin{proof}
 Assumption \ref{assu:domain} and the open mapping theorem yield that $-\Delta^{-1} \in \mathcal L(L^2(\Omega), \mathcal H)$
 and consequently $(-\Delta^{*})^{-1} \in \mathcal L(\mathcal H^*,L^2(\Omega))$.
 Moreover, by definition of $\HH$, the mapping 
 \begin{equation*}
  R: L^2(\Gamma) \to \HH^*,\quad \dual{R u}{v}_{\HH^*,\HH} := -\int_\Gamma u\,\partial_n v\,d\varsigma, \; 
  u \in L^2(\Gamma),\,v\in \HH 
 \end{equation*}
 is linear and continuous. Therefore, 
 \begin{equation}\label{eq:laplaceinvers}
  \eta = (-\Delta^*)^{-1} R u
 \end{equation}
 is the unique solution of \eqref{eq:vweakpoisson}. This immediately implies the a priori estimate 
 with $C = \|(-\Delta)^{-1}\|_{\LL(L^2(\Omega),\HH)} \|R\|_{\LL(L^2(\Gamma),\HH^*)}$.
\end{proof}

\begin{remark}
 The low regularity of the very weak solution implies that 
 the external magnetic field $b = \nabla \eta$ is in general only a distribution and no proper function.
 Note however that, thanks to integration by parts in \eqref{eq:intbypartsode}, only $\eta$ and $u$ appear 
 on the right hand side of \eqref{eq:newtonlorentz}.
\end{remark}

\begin{remark}
 We point out that the magnetic field $b = \nabla \eta$ can be extended outside of $\Omega$ in a divergence-free manner.
 The boundary data $u$, i.e., the control function, can physically be interpreted as a surface current density on $\Gamma$. 
 Naturally, one can, in general, not realize such current density in $L^2(\Gamma)$ in practice so that the numerical results 
 presented in Section \ref{sec:numerics} are rather of theoretical interest.
\end{remark}

Based on the above findings, in particular \eqref{eq:mildsol} and \eqref{eq:laplaceinvers}, we can eliminate 
$E$, $B$, and $\eta$ from the state system to obtain a system of equations in $r$, $p$, and $u$ only.
This gives rise to the following definition:

\begin{definition}[Solution of state system]\label{def:compstate}
 Let the mappings $j$, $F_L$, and $e$ be defined as follows:\\
 1.\ Current density:
  \begin{equation*}
   j: C([0,T];\R^3)^2 \to C([0,T];X)^2,\quad j(r,p)(x,t) := \vektor{- q\,\varphi(x-r(t))v(p(t))}{0},
  \end{equation*}
  2.\ Lorentz force:
  \begin{equation*}
  \begin{aligned}
   &F_L: C([0,T];\R^3)^2 \to C([0,T];X)\\
   &F_L(r,p)(x,t)
   \begin{aligned}[t]
    &:= E(x,t) + \beta(p(t)) \times  B(x,t)\\
    &= \EE(t)\vektor{E_0}{B_0} + \int_0^t \EE(t-\tau) j(r,p)(\tau)\,d\tau\\
    &\quad + \beta(p(t)) \times  \Bigg(\BB(t)\vektor{E_0}{B_0} + \int_0^t \BB(t-\tau) j(r,p)(\tau)\,d\tau\Bigg),
   \end{aligned}
  \end{aligned}
  \end{equation*}
  with the components $\EE$ and $\BB$ of the semigroup $\GG$, see \eqref{eq:GGdef}\\[1mm]
  3.\ State system operator:
  \begin{equation*}
  \begin{aligned}
   & e: C_{\{0\}}^1([0,T]; \R^3)^2 \times L^2(\Gamma) \to C([0,T];\R^3)^2, \quad e(w,z,u):=\vektor{ e_1(w,z,u)}{ e_2(w,z,u)},\\
	&\begin{aligned}
 		e_1(w,z,u)(t) &:= 
 		\begin{aligned}[t]
  			\dot{z}(t) &- q \int_{\Omega} \varphi(x-w(t)-r_0) F_L(w+r_0,z+p_0)(t)\,dx\\
   			&+ q\int_{\Omega} \big((-\Delta^*)^{-1}R u\big) \Big[\nabla\varphi(x-w(t)-r_0) \times\beta(z(t)+p_0)\Big] dx\\ 
  			&- q\int_{\Gamma} u \, \varphi(x-w(t)-r_0) \,\beta(z(t)+p_0)  \times n \, d\varsigma \\ 
 		\end{aligned}\\
 		e_2(w,z,u)(t) &:= \dot{w}(t) - v(z(t)+p_0).
	\end{aligned}
  \end{aligned}
  \end{equation*}
  Then we call a triple $(w,z,u) \in C_{\{0\}}^1([0,T]; \R^3)^2 \times L^2(\Gamma)$ solution of the state system, if it satisfies 
  $e(w,z,u) = 0$.
\end{definition}

We point out that, due to the smoothness assumptions on $\varphi$ in \eqref{eq:smearout} and 
the regularity of the mild solution, see Definition \eqref{def:mildsol}, the mappings $j$, $F_L$, and $e$ indeed possess the 
asserted mapping properties. Note that both PDEs, i.e., Maxwell's equations as well as the Poisson equation, are 
incorporated into this notion of solution by means of the solution operators of the respective PDE in form of 
\eqref{eq:mildsol} and \eqref{eq:laplaceinvers}. Therefore we call the equation $e(w,z,u)=0$ \emph{reduced (state) system}, 
as it only involves the variables $w$, $z$, and $u$. 

With this notion of solution to the state system at hand, we are now in the position to 
state a mathematically rigorous version of the optimal control problem under consideration:
\begin{equation}\label{eq:optconexact}\tag{P}
 \left.
	\begin{aligned}
		\min \quad &\JJ(w+r_0,u)\\
		\text{s.t.}	\quad & w,z \in C_{\{0\}}^1([0,T];\R^3), \ u \in L^2(\Gamma)\\ 
		\quad &e(w,z,u)(t)=0\quad \forall\, t\in [0,T]\\
		&g_i(w(t) + r_0) \leq 0, \quad i = 1, ..., m,\quad \forall\, t\in [0,T].
	\end{aligned}
 \qquad \right\}
\end{equation} 

For the sake of clarity we recall all variables and their meaning in Table \ref{tabvariables}. 
Here and in all what follows, we denote the couple $(w,z)$ by $y$.
For completeness we also list the adjoint variables arising in the upcoming sections in this table.
\begin{table}[h]
\begin{center}
\begin{tabular}{lr}
\toprule
Variable& Name of quantity\\
\midrule
\underline{State variables}& \\
$E$& electric field\\
$B$& magnetic field\\
$r$& position of particle\\
$p$& relativistic momentum of particle\\
$w$& normalized particle position\\
$z$& normalized momentum\\
$y:=(w,z)$&\\
$\eta$ &solution of Poisson equation\\
\underline{Control variable}& \\
$u$ &boundary data of Poisson equation\\
\underline{Adjoint variables}& \\
$\Phi$ & adjoint electric field\\
$\Psi$ & adjoint magnetic field\\
$\varrho$ & adjoint particle position\\
$\pi$ & adjoint relativistic momentum\\
$\omega:= (\varrho,\pi)$& \\
$\chi$ & adjoint Poisson solution\\
$\mu$ & Lagrange multiplier\\
\underline{Further variables}& \\
$j$& electric current\\
$F_L$& Lorentz force\\
$\rho$ &charge density\\
$\gamma$ &Lorentz factor\\
$b$ &external magnetic field\\
$e$ &external electric field\\
$\varphi$ &smeared out delta distribution\\
\bottomrule
\end{tabular}
\captionsetup{font=small} 
 \captionsetup[figure]{labelfont=small}  
\caption{Overview of arising variables}\label{tabvariables}
\end{center}
\end{table}

\section{Analysis of the state equation}\label{sec:state}

We begin the discussion of \eqref{eq:optconexact} with an existence and uniqueness result for the reduced state system.
To be more precise, we prove that, for every $u\in L^2(\Gamma)$, there exists a unique $y\in C_{\{0\}}^1([0,T];\R^3)^2$ 
such that $e(y,u) = 0$. The proof is classical and based on Banach's contraction principle. 
It follows the lines of \cite{SpohnKomech2000} and \cite[Section 2.4]{Spohn2004}, 
where existence and uniqueness is shown for the Abraham model for the case $\Omega = \R^3$ and
without the Poisson equation for the external magnetic field.
Let $u\in L^2(\Gamma)$ be fix but arbitrary. The constraint $e(y,u)=0$ in \eqref{eq:optconexact} is equivalent to
\begin{equation}\label{eq:nlorentzsred}
 \dot y(t) = f(y,u)(t) \quad \forall \ t \in [0,T], \quad y(0) = 0,
\end{equation}
where $f = (f_1, f_2): C([0,T];\R^3)^2 \to C([0,T];\R^3)^2$ is given by
\begin{equation*}
\begin{aligned}
 f_1(w,z,u)(t) &:= 
 \begin{aligned}[t]
  & q \int_{\Omega} \varphi(x-w(t)-r_0) F_L(w+r_0,z+p_0)(t)\,dx\\
  & - q\int_{\Omega} \big((-\Delta^*)^{-1}R u\big) \Big[\nabla\varphi(x-w(t)-r_0) \times\beta(z(t)+p_0)\Big] dx\\ 
  & + q\int_{\Gamma} u \, \varphi(x-w(t)-r_0) \,\beta(z(t)+p_0)  \times n \, d\varsigma 
 \end{aligned}\\
 f_2(w,z,u)(t) &:= v(z(t)+p_0).
\end{aligned}
\end{equation*}
For the rest of this section we suppressed the dependency of $f$ on $u$, as $u$ is fixed throughout this section.
In order to apply the Banach's fixed point theorem, we prove the following

\begin{lemma}\label{lm:Lipred}
 The right hand side  in the reduced system \eqref{eq:nlorentzsred} is globally Lipschitz continuous with respect to $y$ 
 in the following sense
 \begin{equation}\label{eq:lip}
  | f(y_1)(t)-f(y_2)(t)|_2 \leq L\, \|y_1-y_2\|_{C([0,t];\R^3)^2} \quad \forall \,t \in [0,T]
 \end{equation}
 with Lipschitz constant $L \geq 0$.
\end{lemma}

\begin{proof}
 First observe that, by definition of $v$ in \eqref{eq:sigma}, we have
 \begin{equation}\label{eq:sigmaest}
  |v(p)|_2 \leq c, \quad |v'(p)|_F \leq \frac{\sqrt 3}{m_0^q} \quad \forall\, p\in \R^3.
 \end{equation}
 Moreover, \eqref{eq:smearout} implies
 \begin{equation}\label{eq:phiest}
 \begin{aligned}
  \|\varphi(.-r_1)-\varphi(.-r_2)\|_{L^2(\Omega)} &\leq {\sqrt{\frac{4}{3}\,\pi R^3}}\,L_\varphi\,|r_1 - r_2|_2 \quad\forall\, r_1, r_2\in\R^3\\
  \|\varphi(.-r_2)\|_{L^2(\Omega)} &\leq \sqrt{\|\varphi\|_{L^\infty(\R^3)}\|\varphi\|_{L^1(\R^3)}} = \sqrt{C_\varphi} \quad\forall\, r_2\in \R^3,
 \end{aligned}
 \end{equation}
 where $L_\varphi > 0$ denotes the Lipschitz constant of $\varphi$ and $C_\varphi := \max_{x\in \R^3} |\varphi(x)|$. 
 Note that $\varphi$ is globally Lipschitz since it is continuously differentiable and has bounded support.

 The assertion for $f_2$ follows from
 \begin{equation*}
 \begin{aligned}
  & |f_2(y_1)(t) - f_2(y_2)(t)|_2\\
  & = |v(z_1(t) + p_0) - v(z_2(t) + p_0)|_2\\
  & \leq |v'(z_2(t) + p_0 + s(z_1(t)-z_2))|_F |y_2(t) - y_1(t)|_2 \leq \frac{\sqrt 3}{m_0^q}\, \|y_1-y_2\|_{C([0,t];\R^3)^2}.
 \end{aligned}
 \end{equation*}
 To verify the global Lipschitz continuity of $f_1$, we exemplary consider 
 \begin{equation*}
  \hat f(y)(t) := q \int_{\Omega}\varphi(x-w(t)-r_0) \Bigg[\EE(t)\vektor{E_0}{B_0} + 
  \int_0^t \EE(t-\tau) j(w+r_0,z+p_0)(\tau)\,d\tau \Bigg]dx,
 \end{equation*}
 which is one of the terms that arise, if one inserts the definition of $F_L$ into $f_1$. 
 Now let $t\in [0,T]$ and $y_1 = (w_1, z_1), y_2 = (w_2, z_2) \in C([0,t];\R^3)^2$ be arbitrary.
 Using the abbreviations $r_i = w_i + r_0$ and $p_i = z_i + p_0$, $i=1,2$, we obtain 
by means of \eqref{eq:semiexp} that
\begin{align*}
 & |\hat f(y_1)(t) - \hat f(y_2)(t)|_2\\
 &\leq q \Big(\Big\|\EE(t)\vektor{E_0}{B_0}\Big\|_{X}\\
 &\qquad\qquad + \int_0^t \|\EE(t-\tau)  j(r_1,p_1) (\tau) \|_X d \tau \Big)
 \|\varphi(.-r_1(t))-\varphi(.-r_2(t))\|_{L^2(\Omega)}\\
 &\quad + q \int_0^t \|\EE(t-\tau)  j(r_1,p_1) (\tau) - \EE(t-\tau)  j(r_2,p_2) (\tau) \|_X d \tau \,\|\varphi(.-r_2(t))\|_{L^2(\Omega)}\\
 &\leq   q M e^{\omega T} \Big(\|j(r_1,p_1)\|_{L^1([0,t];X\times X)} + \|(E_0,B_0)\|_{X\times X}\Big)
 \sqrt{\pi}\,R\,L_\varphi\,|r_1(t) - r_2(t)|_2\\
 &\quad +q M e^{\omega T} \,\sqrt{C_\varphi}\,\|j(r_1,p_1)- j(r_2,p_2)\|_{L^1([0,t];X\times X)}.
\end{align*}
Concerning the expressions involving $j$, we find by employing \eqref{eq:sigmaest} and \eqref{eq:phiest} that
\begin{equation*}
\begin{aligned}
 & \|j(r_1,p_1)- j(r_2,p_2)\|_{L^1([0,t];X\times X)}\\
 &\quad = q \int_0^t \|\varphi(.-r_1(\tau))v(p_1(\tau)) - \varphi(.-r_2(\tau))v(p_2(\tau))\|_X  d\tau\\
 &\quad\leq q \int_0^t \Big(\|\varphi(.-r_1(\tau))-\varphi(.-r_2(\tau))\|_{L^2(\Omega)} |v(p_1(\tau))|_2\\
 &\quad\qquad\qquad + \|\varphi(.-r_2(\tau))\|_{L^2(\Omega)} |v(p_1(\tau)) - v(p_2(\tau))|_2\Big) d\tau\\
 &\quad\leq q\,T\Big( \sqrt{\pi}\,R\, L_\varphi\,c\,\|r_1 - r_2\|_{C([0,t];\R^3)} 
 + \sqrt{C_\varphi}\,\frac{\sqrt{3}}{m_0^q}\,\|p_1 - p_2\|_{C([0,t];\R^3)} \Big)
\end{aligned}
\end{equation*}
and
\begin{equation}\label{eq:jest}
\begin{aligned}
 \|j(r_1,p_1)\|_{L^1([0,t];X\times X)} 
 &= q\, \int_0^t\|\varphi(.-r_1(\tau))\|_{L^2(\Omega)}\,|v(p_1(\tau))|_2 d\tau\\
 &\leq q\,T\,\sqrt{C_\varphi}\,c.
\end{aligned}
\end{equation}
By inserting these estimates we end up with
\begin{equation*}
\begin{aligned}
 |\hat f(y_1)(t) - \hat f(y_2)(t)| &\leq K\, \big(\|r_1 - r_2\|_{C([0,t];\R^3)} + \|p_1 - p_2\|_{C([0,t];\R^3)}\big)\\
 &\leq \sqrt{2} K\, \|y_1 - y_2\|_{C([0,t];\R^3)^2}
\end{aligned}
\end{equation*}
with a constant $K>0$ independent of $t$, $y_1$, and $y_2$.
The Lipschitz continuity of the remaining parts in $f_1$ can be proven by similar estimates.
\end{proof}

\begin{remark}\label{rem:lip}
 We point out that the Lipschitz constant in \eqref{eq:lip} depends on $u$ so that one should rather write
 \begin{equation*}
  |f(y_1,u)(t)-f(y_2,u)(t)|_2 \leq L(u)\, \|y_1-y_2\|_{C([0,t];\R^3)^2} \quad \forall \,t \in [0,T].
 \end{equation*}
 Of course, the proof of existence of a solution to \eqref{eq:nlorentzsred} for fixed $u$ is not affected by this dependency.
\end{remark}

Based on the Lipschitz-estimate in Lemma \ref{lm:Lipred}, existence and uniqueness can now be shown by Banach's contraction principle. 
The arguments are classical and follow the lines of \cite[Section 2.4]{Spohn2004}. For convenience of the reader we sketch the proof in Appendix \ref{app:stateex}.

\begin{theorem}\label{thm:stateex}
 For all $u \in L^2(\Gamma)$ there exists a unique solution $y \in C_{\{0\}}^1([0,T]; \R^3)^2$ of the reduced system \eqref{eq:nlorentzsred}
 and the following a priori estimate is fulfilled
 \begin{equation*}
  \|y\|_{C^1([0,T];\R^3)^2} \leq C_1\, \|u\|_{L^2(\Gamma)} + C_2
 \end{equation*}
 with a constants $C_1,C_2>0$ independent of $u$ and $y$.
\end{theorem}

\section{Existence of an optimal control}\label{sec:exopt}

With the existence result for the reduced state system in Theorem \ref{thm:stateex} at hand, it is now straightforward 
to establish the existence of a globally optimal control.

\begin{theorem}
 Assume that there is a control $u\in L^2(\Gamma)$ such that the associated state $y = (w,z) \in C_{\{0\}}^1([0,T];\R^3)^2$ 
 satisfies the state constraint $g_i(w(t) + r_0) \leq 0$ for all $i = 1, ..., m$ and all $t\in [0,T]$.
 Then there exists at least one globally optimal control for \eqref{eq:optconexact}.
\end{theorem}

\begin{proof}
 By assumption the feasible set of \eqref{eq:optconexact} is non-empty. 
 Thus there exists a minimizing sequence 
 $\{y_n, u_n\} = \{w_n,z_n,u_n\}\subset C_{\{0\}}^1([0,T];\R^3)^2 \times L^2(\Gamma)$, i.e., $e(y_n,u_n) = 0$, 
 $w_n(t) + r_0\in \tilde\Omega$ for all $t\in [0,T]$, and
 \begin{equation*}
  \JJ(w_n + r_0, u_n) \stackrel{n\to\infty}{\longrightarrow} \inf \eqref{eq:optconexact} =: j \in \R\cup \{-\infty\}.
 \end{equation*}
 From Assumption \ref{assu:further} we deduce
 \begin{equation*}
  \frac{\alpha}{2}\,\|u_n\|_{L^2(\Gamma)}^2 \leq \JJ(w_n + r_0, u_n) - \underline{c_1}\,T - \underline{c_2}
 \end{equation*}
 so that $\{u_n\}$ is bounded in $L^2(\Gamma)$. As $e(y_n,u_n) = 0$, Theorem \ref{thm:stateex} yields 
 the boundedness of $\{y_n\}$ in $H^1([0,T];\R^3)^2$. Consequently, there exist weakly converging subsequences, and 
 w.l.o.g.\ we assume weak convergence of the whole sequences, i.e.
 \begin{equation*}
  u_n  \rightharpoonup u^* \text{ in } L^2(\Gamma) \quad \text{and} \quad
  y_n  \rightharpoonup y^* = (w^*,z^*)\text{ in } H^1(]0,T[;\R^3)^2.
 \end{equation*}
 The compactness of the embedding $H^1(]0,T[;\R^3)^2 \embed C([0,T];\R^3)^2$ then yields strong convergence of $\{y_n\}$ in 
 the maximum-norm so that Lemma \ref{lm:Lipred} and Remark \ref{rem:lip} give
 \begin{equation*}
  \|f(y_n,u^*) - f(y^*,u^*)\|_{C([0,T];\R^3)^2} \leq L(u^*)\,\|y_n-y^*\|_{C([0,T];\R^3)^2} \stackrel{n\to\infty}{\longrightarrow} 0.
 \end{equation*}
 Moreover, the strong convergence of the state in $C([0,T];\R^3)^2$ further implies
 \begin{equation*}
  \|\beta(p_n) - \beta(p^*)\|_{C([0,T];\R^3)} \to 0, 
  \quad \|\varphi(. - r_n) - \varphi(.-r^*)\|_{C([0,T];H^1(\Omega))} \to 0.
 \end{equation*}
 As the control only appears linearly in the state system, 
 these convergences allow to pass to the limit in the reduced state equation in weak form, i.e., for every 
 $v = (v_1, v_2)\in L^2(0,T;\R^3)^2$ there holds
 \begin{equation*}
 \begin{aligned}
  &\int_0^T \dot y^*(t)\cdot v(t)\,dt\\
  &= \lim_{n\to\infty} \int_0^T \dot y_n(t)\cdot v(t)\,dt\\
  &= \lim_{n\to\infty} \int_0^T f(y_n,u_n)(t)\cdot v(t)\,dt\\
  &= \lim_{n\to\infty} \Bigg(\int_0^T f(y_n,u^*)(t)\cdot v(t)\,dt\\
  &\qquad - q\int_{\Omega} \big((-\Delta^*)^{-1}R (u_n - u^*)\big) 
  \int_0^T\Big[\nabla\varphi(x-r_n(t)) \times\beta(p_n(t))\Big] \cdot v_1(t) dt\,dx\\ 
  &\qquad + q\int_{\Gamma} (u_n - u^*) \int_0^T \big[\varphi(x-r_n(t)) \,\beta(p_n(t))  \times n \big]\cdot v_1(t)dt\, d\varsigma\Bigg)\\
  &= \int_0^T f(y^*,u^*)(t)\cdot v(t)\,dt.
 \end{aligned}  
 \end{equation*}  
 Therefore, we obtain
 \begin{equation*}
  \dot y^*(t) = f(y^*,u^*)(t) \quad \text{f.a.a.\ } t\in[0,T].
 \end{equation*}
 Because of $y^* \in C([0,T];\R^3)^2$ the right hand side is continuous such that $y^* \in C^1([0,T];\R^3)^2$. 
 From $y_n \to y^*$ in $C([0,T];\R^3)^2$ we further infer that $y^*(0) = 0$, and consequently 
 $y^*$ coincides with the unique solution of \eqref{eq:nlorentzsred} associated with $u^*$.
 The convergence of the state in $C([0,T];\R^3)^2$ and the continuity of $g_i$, $i = 1, ..., m$, moreover yield
 \begin{equation*}
  g_i(w^*(t) + r_0) \leq 0 \quad\forall \, i = 1, ..., m
  \quad \Leftrightarrow \quad w^*(t) + r_0 \in \tilde\Omega
 \end{equation*}
 for all $t\in [0,T]$ such that the state constraint is also fulfilled in the limit. Therefore, the couple $(y^*, u^*)$ 
 fulfills all constraints in \eqref{eq:optconexact}.

 Finally, the strong convergence of $\{y_n\}$ in $C([0,T];\R^3)^2$, the weak convergence of $\{u_n\}$ in $L^2(\Gamma)$, 
 and the weak lower semicontinuity of $\|.\|_{L^2(\Gamma)}^2$ allow to pass to the limit in the objective:
 \begin{equation*}
 \begin{aligned}
  j &= \lim_{n\to\infty} \JJ(w_n + r_0, u_n) \\
  &\geq \lim_{n\to\infty} \Big(\int_0^T J_1(w_n(t)+r_0)\,dt + J_2(w_n(T)+r_0)\Big) + 
  \liminf_{n\to\infty} \frac{\alpha}{2} \int_\Gamma u_n^2\,d\varsigma\\
  &\geq \JJ(w^* + r_0, u^*),
 \end{aligned}
 \end{equation*}
 which implies the optimality of $(y^*,u^*)$.
\end{proof}

\section{First-order necessary optimality conditions}\label{sec:nec}

For the rest of the paper, we slightly change the functional analytical framework of the optimal control problem 
under consideration. To be more precise, we weaken the regularity of the state space in order to obtain a more regular 
adjoint state and treat the state as a function in
$$Y = \{ y \in H^1(]0,T[;\R^3)^2 : y(0) = 0\}.$$
Thus the mapping associated with the reduced state system becomes 
$e: Y \times L^2(\Gamma) \to Z = L^2(]0,T[;\R^3)^2$, with a slight abuse of notation still denoted by $e$. 
It is easily seen that this modification does not affect the above analysis, in particular the proof of existence of 
an optimal control, since the state is treated as a function in $H^1(]0,T[;\R^3)^2$ there anyway.
Note that $H^1(]0,T[;\R^3)^2 \embed C([0,T];\R^3)^2$ so that the mappings $j$ and $F_L$ from Definition \ref{def:compstate} 
are still well-defined.

\begin{remark}
 If a couple $(y,u)\in Y \times L^2(\Gamma)$ satisfies the constraint $e(y,u) = 0$, i.e., 
 \begin{equation*}
   \dot y(t) = f(y,u)(t) \quad \text{f.a.a.\ } t \in [0,T], \quad y(0) = 0,
 \end{equation*}
 then $f(y,u) \in C([0,T];\R^3)^2$ implies $y\in C^1([0,T];\R^3)^2$ so that $y$ coincides with the 
 unique solution of \eqref{eq:nlorentzsred} from Theorem \ref{thm:stateex}. In other words, the treatment of \eqref{eq:optconexact} 
 in the weaker state space $Y$ does not affect the regularity of the optimal state.
\end{remark}

\subsection{The linearized state equation}

We start the derivation of a qualified optimality system by the analysis of the linearized reduced state system.

\begin{lemma}\label{lem:deriv}
The reduced form $e$ is continuously Fr\'echet-differentiable from $Y \times L^2(\Gamma)$ to $Z$.
Its partial derivatives at $(y,u) = (w,z,u)\in Y \times L^2(\Gamma)$ in direction 
$(\phi,h) = (\phi_r, \phi_p,h)\in Y \times L^2(\Gamma)$ are given by
\begin{equation*}
\begin{aligned}
 \Big(\frac{\partial e_1}{\partial y}(y,u) \phi\Big)(t) &= \dot\phi_p(t) - \Big(\ddp{f_1}{y}(y,u)\phi\Big)(t), \\
 \Big(\frac{\partial e_2}{\partial y}(y,u) \phi\Big)(t) &= \dot\phi_r(t) - \Big(\ddp{f_2}{y}(y,u)\phi\Big)(t),\\
 \Big( \frac{\partial e_1}{\partial u}(y,u) h\Big) (t) &= - \Big(\ddp{f_1}{u}(y,u)h\Big)(t), \quad
 \Big( \frac{\partial e_1}{\partial u}(y,u) h\Big) (t) = 0
\end{aligned}
\end{equation*} 
with 
\begin{align*}
 \Big(\ddp{f_1}{u}(y,u)h\Big)(t) &= q\int_{\Gamma} h \varphi(x-r(t)) \beta(p(t)) \times n \, d \varsigma\\
 &\quad -q\int_{\Omega}(-\Delta^*)^{-1} R h [\nabla \varphi(x-r(t)) \times \beta(p(t))] \, dx,\\
 \Big(\ddp{f_2}{y}(y,u)\phi\Big)(t) &= v'(p(t))\phi_p(t),\\
 \intertext{and}
 \Big(\ddp{f_1}{y}(y,u)\phi\Big)(t) &= 
  - q \int_{\Omega} \big[\nabla\varphi(x-r(t))\cdot \phi_r(t)\big] F_L(r,p)(t)\,dx\\
  &\quad + q \int_{\Omega} \varphi(x-r(t)) \big(\partial_y F_L(r,p)\phi\big)(t)\, dx\\
  &\quad + q \int_{\Gamma} u 
  \begin{aligned}[t]
   \Big[ & \varphi(x-r(t))\,\beta'(p(t)) \phi_p(t)\\
   & - \big[\nabla\varphi(x-r(t)) \cdot \phi_r(t)\big] \beta(p(t))\Big] \times n \, d\varsigma
  \end{aligned}\\   
  &\quad + q\int_{\Omega} \big((-\Delta^*)^{-1} R u\big)
  \begin{aligned}[t]
   \Big[ & \nabla^2\varphi(x-r(t)) \phi_r(t) \times \beta(p(t)) \\
   & - \nabla\varphi(x-r(t)) \times \beta'(p(t)) \phi_p(t) \Big] dx,
  \end{aligned}
\end{align*} 
with $r = w + r_0$, $p = z + p_0$, the derivative of the Lorentz force term $F_L$  
\begin{equation*}
\begin{aligned}
 \big(\partial_y F_L(r,p)\phi\big)(t) &= \big(\partial_r F_L(r,p)\phi_r + \partial_p F_L(r,p)\phi_p\big)(t)\\
 &= \int_0^t \EE(t-\tau) \big(j'(r,p)(\tau) \phi(\tau)\big)\,d\tau\\
 &\quad + \beta(p(t)) \times \int_0^t \BB(t-\tau) \big(j'(r,p)(\tau)\phi(\tau)\big)\,d\tau\\
 &\quad + \beta'(p(t))\phi_p(t) \times \Bigg(\BB(t)\vektor{E_0}{B_0} + \int_0^t \BB(t-\tau) j(r,p)(\tau)\,d\tau\Bigg)
\end{aligned}
\end{equation*}
and $j'$ as given in \eqref{eq:jprime}.
\end{lemma}

\begin{proof}
 As a linear and bounded operator the time derivative is clearly continuously Fr\'echet-differentiable for $H^1(]0,T[;\R^3)$ 
 to $L^2(]0,T[;\R^3)$. All nonlinear Nemyzki-operators involved in $f$ are differentiated in 
 spaces of continuous functions. Because of its slightly non-standard structure, we exemplary study the Fr\'echet-differentiability of 
 $r\mapsto \nabla\varphi(.-r)$ from $C([0,T];\R^3)$ to $C([0,T];L^2(\Omega))$:
 \begin{equation*}
 \begin{aligned}
  & \|\nabla\varphi(. - (r+\phi_r)) - \nabla\varphi(.-r) - \nabla^2 \varphi(. - r)\cdot \phi_r\|_{C([0,T];L^2(\Omega))}^2\\
  &\quad = \max_{t\in [0,T]} 
  \Big(\int_\Omega |\nabla \varphi(x-r(t) - \phi_r(t)) - \nabla \varphi(x-r(t)) - \nabla^2\varphi(x-r(t))\phi_r(t)|^2\,dx\\
  &\quad = \max_{t\in [0,T]} 
  \int_\Omega \Big| \int_0^1 \nabla^2\varphi\big(x - r(t) - \theta\phi_r(t)\big)\phi_r(t) d\theta - \nabla^2\varphi(x-r(t))\phi_r(t)\Big|^2\,dx\\
  &\quad \leq \max_{t\in [0,T]} \int_\Omega \Big|\int_0^1 L_{\varphi,2} \,\theta\,|\phi_r(t)|^2 d\theta\Big|^2 dx
  = \frac{1}{4}\, L_{\varphi,2}^2\, |\Omega|\, \|\phi_r(t)\|_{C([0,T];\R^3)}^4,
 \end{aligned}
 \end{equation*}  
 where $L_{\varphi,2}$ denotes the Lipschitz constant of $\nabla^2 \varphi$.
 This gives the partial differentiability of $f$ w.r.t.\ $y$. As $u$ only appears linearly, $f$ is moreover partially differentiable 
 w.r.t.\ $u$. Furthermore, one readily verifies that these partial derivatives are continuous in $(y,u)$. 
 Therefore, \cite[Theorem 3.7.1]{cartan1983} gives the continuous Fr\'echet-differentiability of $e$.
\end{proof}

\begin{lemma}\label{lm:existenceredsys}
 Let $(y,u) \in Y \times L^2(\Gamma)$ be given. Then
 for every $h\in Z$ there exists a unique solution $\phi = (\phi_r, \phi_p) \in Y$ of the linearized equation
 \begin{equation}\label{eq:linode}
  \frac{\partial e}{\partial y}(y,u) \phi = h.
 \end{equation}
\end{lemma}

\begin{proof}
 In view of Lemma \ref{lem:deriv}, \eqref{eq:linode} is equivalent to 
 \begin{equation*}
  \vektor{\dot \phi_p(t)}{\dot \phi_r(t)} = \Big(\ddp{f}{y}(y,u)\phi\Big)(t) + h(t) \quad \text{f.a.a.\ } t\in [0,T], \quad \phi(0) = 0
 \end{equation*}
 with $\partial_y f(y,u)\phi = (\partial_y f_1(y,u)\phi, \partial_y f_2(y,u)\phi)$. 
 As in the proof of Theorem \ref{thm:stateex}, existence and uniqueness of the equivalent integral equation, given by
 \begin{equation*}
  \vektor{ \phi_p(t)}{ \phi_r(t)} = \int_0^t \Big[\Big(\ddp{f}{y}(y,u)\phi\Big)(\tau) + h(\tau)\Big] d \tau,
 \end{equation*}
 can again be proven by Banach's contraction principle, provided that there is a constant $C>0$ such that
 \begin{equation*}
  \Big|\Big(\ddp{f}{y}(y,u)\phi\Big)(t)\Big|_2 \leq C\,\|\phi\|_{C([0,t];\R^3)^2} \quad \forall\, t\in [0,T],
 \end{equation*}
 cf.\ \eqref{eq:lip}. (Note in this context that $\phi \mapsto \partial_yf(y,u)\phi$ is a linear mapping so that 
 Lipschitz continuity is equivalent to boundedness.) The latter inequality however can be verified by estimates 
 similar to the proof of Lemma \ref{lm:Lipred}. 
\end{proof}

\subsection{KKT conditions}

Having established the differentiability of the reduced state system, we are now in the position to 
derive first-order optimality system in qualified form, i.e., Karush-Kuhn-Tucker (KKT) conditions involving Lagrange multipliers 
associated with the constraints in \eqref{eq:optconexact}. 
To this end, let $(y^*,u^*) = (w^*, z^*,u^*)\in Y \times L^2(\Gamma)$ be a arbitrary local optimum of \eqref{eq:optconexact}.
As before, we set $r^* = w^* + r_0$ and $p^* = z^* + p_0$ in all what follows.
It is known that the existence of Lagrange multipliers requires certain constraint qualifications, see 
e.g.\ \cite{ZoweKurcyusz1979}. In our case, one of these, namely the surjectivity of $\partial_y e(y^*,u^*)$, was established in 
Lemma \ref{lm:existenceredsys}. 
However, we need an additional condition to obtain a Lagrange multiplier for the pointwise state constraint in \eqref{eq:optconexact}, too.

\begin{assumption}[Linearized Slater condition]\label{assu:slater}
 We assume that there is a function $\hat h\in L^2(\Gamma)$ so that 
 \begin{equation}\label{eq:lSlater}
 	g_i(r^*(t)) + g_i'(r^*(t)) \hat\phi_r(t) < 0 \quad \forall \, t \in [0,T],\; i= 1,...,m,
 \end{equation}
 where $\hat\phi = (\hat\phi_r,\hat\phi_p)\in Y$ is the solution to \eqref{eq:linode} for $h = \hat h$.
\end{assumption}

Note that the Nemyzki operators associated with $g_1, ..., g_m$ are Fr\'echet-differentiable from $C([0,T];\R^3)$ 
to $C([0,T])$ by Assumption \ref{assu:further}. The same holds for the functions $J_1$ and $J_2$ within the objective.

Given that Assumption \ref{assu:slater} is fulfilled, one can establish the existence of Lagrange multipliers, 
see for instance \cite[Section 1.7.3.4]{HinzePinnauUlbrichUlbrich09}. To be more precise, under 
Assumption \ref{assu:slater} there exists
$(\pi,\varrho,\lambda) \in Z \times C([0,T]; \R^m)^*$ such that the following KKT conditions are satisfied:
\begin{subequations}\label{seq:kktsys}
\begin{gather}
 e(y^*,u^*)(t) = 0\quad \forall\, t\in [0,T] \label{seq:kktstate}\\[2mm]
 \frac{\partial e}{\partial y}(y^*,u^*)^* \vektor{\pi}{\varrho} 
 -\frac{\partial \JJ}{\partial y}(r^*,u^*)-\vektor{g'(r^*)^* \lambda}{0} =0 \quad \text{in } Y^* \label{seq:kktad}\\
 \frac{\partial \JJ}{\partial u}(r^*,u^*) + \frac{\partial e}{\partial u}(y^*,u^*)^* \vektor{\pi}{\varrho} = 0 
 \quad \text{in } L^2(\Gamma) \label{seq:kktgrad}\\[2mm]
 \begin{aligned}
  g_i(r^*(t)) &\leq 0 \quad \forall\, t\in [0,T],\\
  \lambda_i &\geq 0, \quad \langle \lambda_i, g_i(r^*)\rangle_{C([0,T])^*, C([0,T])}=0, \quad i = 1, ...m. 
 \end{aligned}\label{seq:kktcomp}
\end{gather}
\end{subequations}
Herein the inequality $\lambda_i \geq 0$ is to be understood in a distributional sense, i.e., 
$\dual{\lambda_i}{v} \geq 0$ for all $v\in C([0,T])$ with $v(t) \geq 0$ for all $t\in [0,T]$. 
Moreover, we set $g := (g_1, ..., g_m)$ and denote by $g'$ the associated Jacobian.

For the rest of this section, we aim to transfer \eqref{seq:kktad} to an adjoint system 
and to evaluate the gradient equation in \eqref{seq:kktgrad}. We start with \eqref{seq:kktad}, which in 
variational form  reads as follows
\begin{equation}\label{mul:adjeq}
\begin{aligned}
 \int_0^T \Big[ \pi(t) \cdot \Big(\ddp{e_1}{y}(y^*, u^*)\phi\Big)(t)
 + \varrho(t) \cdot \Big(\ddp{e_2}{y}(y^*, u^*)\phi\Big)(t)\Big] dt \qquad\quad\qquad &\\
 - \sdual{\ddp{\JJ}{r}(r^*, u^*)}{\phi_r}_{Y^*,Y} - \dual{\lambda}{g'(r^*)\phi_r}_{C([0,T];\R^m)^*, C([0,T];\R^m)} = 0 \quad&\\
 \forall\, \phi \in Y.&
\end{aligned}
\end{equation}
By employing Lemma \ref{lem:deriv} we find for the first term in \eqref{mul:adjeq}
\begin{equation*}
\begin{aligned}
 & \int_0^T \pi(t) \cdot \Big(\ddp{e_1}{y}(y^*, u^*)\phi\Big)(t) \,dt\\
 & \; = \int_0^T \dot\phi_p(t)\cdot \pi(t)\,dt - q\,I_L(\pi, y^*, \phi)\\
 &\quad + q \int_0^T \phi_r(t)\cdot 
 \begin{aligned}[t]
  \Bigg( & \int_{\Omega} \big[F_L(r^*,p^*)(t)\cdot \pi(t)\big] \nabla\varphi(x-r^*(t))\,dx\\
  & + \int_{\Omega}\big((-\Delta^*)^{-1} R u^*\big) \Big[\beta(p^*(t)) \times  \nabla^2\varphi(x-r^*(t)) \pi(t)\Big] dx\\
  & - \int_{\Gamma} u^* \big[\nabla\varphi(x-r^*(t)) \cdot \pi(t)\big]\big( n \times  \beta(p^*(t))\big) \, d\varsigma \Bigg) dt
 \end{aligned}\\
 &\quad - q \int_0^T \phi_p(t)\cdot 
 \begin{aligned}[t]
  \Bigg( & \int_{\Omega}\big((-\Delta^*)^{-1} R u^*\big) \Big[\beta'(p^*(t))\pi(t) \times  \nabla\varphi(x-r^*(t))\Big] dx\\
  & - \int_{\Gamma} u^* \,\varphi(x-r^*(t))\big( n \times  \beta'(p^*(t)) \pi(t)\big) \, d\varsigma \Bigg) dt, 
 \end{aligned}
\end{aligned}
\end{equation*}
where $I_L(\pi, y^*,\phi)$ is defined by
\begin{equation*}
 I_L(\pi, y^*,\phi)
 := \int_0^T \pi(t) \cdot \int_{\Omega} \varphi(x-r^*(t)) \big(\partial_y F_L(r^*,p^*)\phi\big)(t)\, dx\,dt.
\end{equation*}
In view of Lemma \ref{lem:deriv}, applying Fubini's theorem to this expression leads to
\begin{equation*}
\begin{aligned}
 & I_L(\pi, y^*,\phi)\\
 &= \begin{aligned}[t]
 \int_0^T\int_0^t \int_\Omega \Big[ & \EE(t-\tau) \big(j'(r^*,p^*)(\tau)\phi(\tau)\big)\\
 &+ \beta(p^*(t)) \times \BB(t-\tau) \big(j'(r^*,p^*)(\tau) \phi(\tau)\big)\Big] \cdot \varphi(x-r^*(t))\pi(t) \,dx\,d\tau\,dt \\
 \end{aligned}\\
 &\quad + \int_0^T\int_\Omega \varphi(x-r^*(t))\, \beta'(p(t))\phi_p(t) \times B^*(t) \cdot \pi(t)\,dx\,dt\\
 &= \int_0^T \phi(t) \cdot \int_\Omega j'(r^*,p^*)(t)^\top \int_t^T \GG(\tau-t)^*\kappa(r^*,p^*,\pi)(\tau)\,d\tau\,dx\,dt\\
 &\quad + \int_0^T \phi_p(t) \cdot \int_\Omega B^*(t) \times \varphi(x-r^*(t))\,\beta'(p^*(t))\pi(t)\, dx\,dt,
\end{aligned}
\end{equation*}
where we abbreviated  
$$B^*(t) := \BB(t)\vektor{E_0}{B_0} + \int_0^t \BB(t-\tau) j(r,p)(\tau)\,d\tau$$
and set
\begin{equation*}
\begin{aligned}
 &\kappa: \R^3\times \R^3\times \R^3 \to \R^3\times \R^3\\
 &\kappa(r,p,\pi) := \vektor{\varphi(x-r)\,\pi}{\pi \times \varphi(x-r)\,\beta(p)}.
\end{aligned}
\end{equation*}
Moreover let us define
\begin{equation}\label{eq:adEBfields}
 \vektor{\Phi(t)}{\Psi(t)} :=  \int_t^T \GG(\tau-t)^*\kappa(r^*,p^*,\pi)(\tau)\,d\tau
\end{equation}
Since $-i\AA$ is self-adjoint, the theorem of Stone implies that $\GG(t)^*$ is the semigroup generated by the adjoint operator 
\begin{equation*}
 \AA^*: X\times X \to X\times X,\quad 
 \AA^* = \left(\begin{matrix}
  0 & \curl\\
  -\curl & 0\\
 \end{matrix} \right)
\end{equation*}
with domain $D(\AA^*) = D(\AA) = \Hc$. Thus $(\Phi, \Psi)\in C([0,T];X)^2$ is the mild solution of the following backward-in-time problem:
\begin{equation}\label{eq:admax}
\begin{aligned}
 -\frac{\partial }{\partial t} \vektor{\Phi(t)}{\Psi(t)} 
 + \AA^* \vektor{\Phi(t)}{\Psi(t)} &= 
 \begin{pmatrix}
  \varphi(\,.\,-r^*(t))\,\pi(t)\\[1mm]
  \pi(t) \times \varphi(\,.\,-r^*(t))\,\beta(p^*(t))
 \end{pmatrix}\\
 \Phi(T) = \Psi(T) &= 0.
\end{aligned}
\end{equation}
By setting $\omega := (\varrho,\pi)$ and summarizing the above transformations, we obtain for the first two addends in \eqref{mul:adjeq}
\begin{multline*}
  \int_0^T \Big[\pi(t) \cdot \Big(\ddp{e_1}{y}(y^*, u^*)\phi\Big)(t) 
  + \varrho(t) \cdot \Big(\ddp{e_2}{y}(y^*, u^*)\phi\Big)(t)\Big] dt \\
  = \int_0^T \dot\phi(t)\cdot \omega(t)\,dt + \int_0^T \phi(t) \cdot A(y^*,u^*,\omega)(t)\,dt
 \end{multline*}
with 
\begin{equation}\label{eq:Adef}
\begin{aligned}
 & A(y^*,u^*,\omega)(t) = \vektor{A_r(y^*,u^*,\omega)(t)}{A_p(y^*,u^*,\omega)(t)}\\
 &\quad := q \begin{pmatrix}
  \int_{\Omega} \big[F_L(r^*,p^*)(t)\cdot \pi(t)\big] \nabla\varphi(x-r^*(t))\,dx\\[2mm]
  - \int_\Omega B^*(t) \times \varphi(x-r^*(t))\,\beta'(p^*(t))\pi(t)\, dx 
 \end{pmatrix}\\
 &\qquad + 
 q \begin{pmatrix}
  \int_{\Omega} \eta^* \Big[\beta(p^*(t)) \times  \nabla^2\varphi(x-r^*(t)) \pi(t)\Big] dx\\[2mm]
  - \int_{\Omega}\eta^* \Big[\beta'(p^*(t))\pi(t) \times  \nabla\varphi(x-r^*(t))\Big] dx
 \end{pmatrix}\\
 &\qquad + 
 q \begin{pmatrix}
  -\int_{\Gamma} u^* \big[\nabla\varphi(x-r^*(t)) \cdot \pi(t)\big]\big( n \times  \beta(p^*(t))\big) \, d\varsigma\\[2mm]
  \int_{\Gamma} u^* \,\varphi(x-r^*(t))\big( n \times  \beta'(p^*(t)) \pi(t)\big) \, d\varsigma
 \end{pmatrix}\\
  &\qquad + q^2
  \begin{pmatrix}
   -\int_\Omega v(p^*(t)) \big[\nabla\varphi(x - r^*(t)) \cdot \Phi(t)\big] dx\\[2mm]
    \int_\Omega \varphi(x-r^*(t))\, v'(p^*(t))\Phi(t)\, dx
  \end{pmatrix}    
  - \begin{pmatrix}
  0 \\  v'(p^*(t))\varrho(t)
 \end{pmatrix},
\end{aligned}
\end{equation}
where $\eta^* = (-\Delta^*)^{-1}\,R\,u^*$.
Thus the adjoint equation \eqref{mul:adjeq} becomes
\begin{equation}\label{eq:admitmass}
\begin{aligned}
 \int_0^T \dot\phi(t)\cdot \omega(t)\,dt + \int_0^T \phi(t) \cdot A(y^*,u^*,\omega)(t)\,dt 
 - \sdual{\ddp{\JJ}{r}(r^*, u^*)}{\phi_r}_{Y^*,Y}&\\
 - \dual{\lambda}{g'(r^*)\phi_r}_{C([0,T];\R^m)^*, C([0,T];\R^m)} = 0 
 \quad\forall \, \phi \in &\, Y.
\end{aligned}
\end{equation}
By the Riesz representation theorem $\lambda \in C([0,T];\R^m)^*$ can be identified with a function of bounded variations. 
This leads to the following result, whose detailed proof is given in Appendix \ref{app:nbv}.

\begin{lemma}\label{lem:nbv}
 The adjoint particle position $\varrho$ and the adjoint momentum $\pi$ satisfy 
 $\varrho \in \mathrm{BV}([0,T];\R^3)$ and $\pi \in W^{1,\infty}(]0,T[;\R^3)$.
 Together with a function $\mu\in \mathrm{NBV}([0,T];\R^m)$ they fulfill the following ODEs backward in time:
 \begin{align}
  -\dot\pi(t) &= - A_p(y^*, u^*, \varrho, \pi)(t) \quad \text{a.e.\ in } ]0,T[ \label{eq:piode}\\
  \pi(T) &= 0 \label{eq:piend}\\
  - \dot\varrho(t) &= - A_r(y^*, u^*, \varrho, \pi)(t) + \nabla J_1(r^*(t)) - g'(r^*(t))^\top \dot\mu(t) 
  \quad \text{a.e.\ in } ]0,T[ \label{eq:rhoode}\\
  \varrho(T) &= \nabla J_2(r^*(T)). \label{eq:rhoend}
 \end{align}
 In addition, $\mu$ is monotone increasing and satisfies
 \begin{equation*}
  \int_0^T g(r^*(t)) \cdot d\mu(t) = 0.
 \end{equation*}
 Moreover, $\varrho$ only admits finitely many points of discontinuity $t_1, ..., t_\ell$ in $]0,T[$, at each of which 
 \begin{equation}\label{eq:rhojump}
  \varrho(t_i) - \lim_{\varepsilon\searrow 0}\varrho(t_i - \varepsilon)
  = g'(r^*(t_i))^\top \big(\lim_{\varepsilon \searrow 0} \mu(t_i-\varepsilon) - \mu(t_i)\big), \quad
  i = 1, ..., \ell,
 \end{equation}
 holds true.
\end{lemma}

Next we turn to the gradient equation \eqref{seq:kktgrad}.
Focusing on the second addend in \eqref{seq:kktgrad}, we obtain by means of Lemma \ref{lem:deriv} that
\begin{equation*}
\begin{aligned}
 &\int_\Gamma \Big(\frac{\partial e}{\partial u}(y^*,u^*)^* \omega\Big)\phi_u \,d\varsigma\\
 &\quad = \int_0^T  \vektor{\pi(t)}{\varrho(t)} \cdot \Big(\frac{\partial e}{\partial u}(y^*,u^*) \phi_u\Big)(t)\,dt\\
 &\quad = q \int_0^T \pi(t) \cdot 
 \int_{\Omega} \big((-\Delta^*)^{-1} R \phi_u\big) \big[\nabla\varphi(x-r^*(t)) \times \beta(p^*(t))\big] \, dx \, dt\\ 
 &\qquad  -q \int_0^T \pi(t) \cdot \int_{\Gamma} \phi_u \varphi(x-r^*(t)) \beta(t,p^*(t)) \times n \, d\varsigma \, dt\\
 &\quad = \int_\Gamma \phi_u  \,q \int_0^T 
 \begin{aligned}[t]
  \Big( & R^* (-\Delta)^{-1} \Big[\big(\nabla\varphi(x-r^*(t)) \times \beta(p^*(t))\big) \cdot \pi(t)\Big]\\
  & - \big[\varphi(x-r^*(t)) \beta(p^*(t)) \times n\big] \cdot \pi(t) \Big) dt\,d\varsigma.
 \end{aligned}  
\end{aligned}
\end{equation*}
Let us define the adjoint Poisson solution by
\begin{equation*}
 \chi(t) :=  - \Delta^{-1} \Big[\big(\nabla \varphi(\,.\,-r^*(t)) \times \beta(p^*(t))\big) \cdot \pi(t)\Big] \in \HH.
\end{equation*}
Note that the regularity w.r.t.\ time carries over from $\pi$ to $\chi$ so that 
\begin{equation*}
 \chi \in W^{1,\infty}(]0,T[; \HH).
\end{equation*}
Then, in view of $\partial_u \JJ(r^*,u^*) = \alpha u^* \in L^2(\Gamma)$ and $R^* = - \partial_n : \HH \to \L^2(\Gamma)$, 
the gradient equation \eqref{seq:kktgrad} becomes
\begin{multline*}
 \int_{\Gamma} \Big( q \int_0^T \Big[-\partial_n  \chi(t) - \big[\varphi(x-r^*(t)) \beta(p^*(t)) \times n\big] \cdot \pi(t) \Big] dt
 + \alpha u^* \Big)   \phi_u \, d\varsigma = 0\\
 \forall\, \phi_u \in L^2(\Gamma)
\end{multline*}
and the fundamental lemma of calculus of variations yields
\begin{equation*}
 u^*(x) = \frac{q}{\alpha} \int_0^T \Big[\partial_n  \chi(x,t) + \big[\varphi(x-r^*(t)) \beta(p^*(t)) \times n\big] \cdot \pi(t) \Big] dt 
 \quad \text{a.e.\ on } \Gamma.
\end{equation*}

Summarizing the results we have, thus, derived the following first-order necessary optimality conditions for \eqref{eq:optconexact}:

\begin{theorem}[KKT conditions]\label{thm:kkt}
Let $u^* \in L^2(\Gamma)$ be a locally optimal boundary control with associated states 
$(E^*,B^*,\eta^*,r^*,p^*) \in C([0,T];X)^2 \times L^2(\Omega) \times C^1([0,T]; \R^3)^2$. 
Assume further that the linearized Slater condition in Assumption \ref{assu:slater} is fulfilled. Then there exist adjoint states 
\begin{equation*}
 (\Phi,\Psi,\chi,\varrho,\pi) \in C([0,T];X)^2 \times W^{1,\infty}(]0,T[; \HH) \times \mathrm{BV}([0,T];\R^3) \times W^{1,\infty}([0,T];\R^3)
\end{equation*}
and a Lagrange multiplier $\mu \in \mathrm{NBV}([0,T];\R^m)$ so that following optimality system is fulfilled:\\[1mm]
\underline{State equations:}\\[1mm]
Maxwell equations:
\begin{equation*}
\begin{aligned}
 \frac{\partial }{\partial t} \vektor{E^*(t)}{B^*(t)} 
 + \AA \vektor{E^*(t)}{B^*(t)} &= \vektor{-q\, \varphi(\,.\,-r^*(t))\, v(p^*(t))}{0} \quad \text{a.e.\ in } [0,T] \\
 E^*(0) = E_0, \quad B^*(0) &= B_0
\end{aligned}
\end{equation*}
Newton-Lorenz equation:
\begin{equation*}
\begin{aligned}
 \dot p^*(t) &= q 
 \begin{aligned}[t]
  \Big( & \int_{\Omega} \varphi(x-r^*(t))\big( E^*(t) + \beta(p^*(t)) \times B^*(t) \big) dx\\
  & + \int_{\Gamma} u^*\, \varphi(x-r^*(t))\, \beta(p^*(t))  \times n\, d\varsigma\\
  & - \int_{\Omega} \eta^* \big[\nabla\varphi(x-r^*(t)) \times \beta(p^*(t))\big] dx \Big) \quad \forall \, t \in [0,T]
 \end{aligned}\\
 p^*(0) &= 0\\
 \dot r^*(t) &= v(p^*(t))  \quad \forall \, t \in [0,T]\\
 r^*(0) &= 0
\end{aligned}
\end{equation*}
Poisson's equation in very weak form:
\begin{equation*}
 \int_{\Omega} \eta^* \Delta v \, dx = \int_{\Gamma} u^* \, \partial_n v \, d\varsigma \quad \forall \,  v \in \HH
\end{equation*}
\underline{Adjoint equations:}\\[1mm]
Adjoint Maxwell equations:\\
\begin{equation*}
\begin{aligned}
 - \frac{\partial }{\partial t} \vektor{\Phi(t)}{\Psi(t)} 
 + \AA^* \vektor{\Phi(t)}{\Psi(t)} &= 
 \begin{pmatrix}
  \varphi(\,.\,-r^*(t))\,\pi(t)\\[1mm]
  \pi(t) \times \varphi(\,.\,-r^*(t))\,\beta(p^*(t))
 \end{pmatrix}\\
 \Phi(T) = \Psi(T) &= 0
\end{aligned}
\end{equation*}
Adjoint ODE system:
\begin{equation}
\begin{aligned}
 - \dot\pi(t) &= 
 q \int_\Omega B^*(t) \times \varphi(x-r^*(t))\,\beta'(p^*(t))\pi(t)\, dx \\
 &\quad + q \int_{\Omega}\eta^* \big[\beta'(p^*(t))\pi(t) \times  \nabla\varphi(x-r^*(t))\big] dx\\
 &\quad - q \int_{\Gamma} u^* \,\varphi(x-r^*(t))\big( n \times  \beta'(p^*(t)) \pi(t)\big) \, d\varsigma\\
 &\quad - q^2 \int_\Omega \varphi(x-r^*(t))\, v'(p^*(t))\Phi(t)\, dx + v'(p^*(t))\varrho(t) \\ 
 &\hspace*{7cm} \text{a.e.\ in } ]0,T[\\
 \pi(T) &= 0
\end{aligned}
\end{equation}
\begin{equation}\label{eq:rhoode2}
\begin{aligned}
 - \dot\varrho(t)
 &= -q \int_{\Omega}\eta^* \big[\beta(p^*(t)) \times  \nabla^2\varphi(x-r^*(t)) \pi(t)\big] dx\\
 &\quad + q \int_{\Gamma} u^* \big[\nabla\varphi(x-r^*(t)) \cdot \pi(t)\big]\big( n \times  \beta(p^*(t))\big) \, d\varsigma\\
 &\quad - q\int_{\Omega} \big[\big(E^*(t) + \beta(p^*(t)) \times  B^*(t)\big)\cdot \pi(t)\big] \nabla\varphi(x-r^*(t)) \,dx\\
 &\quad + q^2\int_\Omega v(p^*(t)) \big[\nabla\varphi(x - r^*(t)) \cdot \Phi(t)\big] dx\\
 &\quad + \nabla J_1(r^*(t)) - g'(r^*(t))^\top \dot\mu(t) \qquad \text{a.e.\ in } ]0,T[\\
 \varrho(T) &= \nabla J_2(r^*(T))
\end{aligned}
\end{equation}
Jump conditions:
\begin{equation}\label{eq:rhojump2}
 \varrho(t_i) - \lim_{\varepsilon\searrow 0}\varrho(t_i - \varepsilon)
 = g'(r^*(t_i))^\top \big(\lim_{\varepsilon \searrow 0} \mu(t_i-\varepsilon) - \mu(t_i)\big), \quad
 i = 1, ..., \ell,
\end{equation}
Adjoint Poisson equation:
\begin{equation*}
\begin{aligned}
 -\Delta \chi(x,t) &=  \big(\nabla \varphi(x-r^*(t)) \times \beta(p^*(t))\big) \cdot \pi(t) 
 & & \text{ f.a.a.\ } (x,t)\in\; ]0,T[\times \Omega \\
 \chi(x,t) &= 0 & & \text{ f.a.a.\ } (x,t)\in \;]0,T[\times \Gamma
\end{aligned}
\end{equation*} 
\underline{Gradient equation:}
\begin{equation*}
 u^*(x) = \frac{q}{\alpha} \int_0^T \Big[\partial_n  \chi(x,t) + \big[\varphi(x-r^*(t)) \beta(p^*(t)) \times n\big] \cdot \pi(t) \Big] dt 
 \quad \text{a.e.\ on } \Gamma
\end{equation*}
\underline{Complementary relations:}
\begin{equation*}
\begin{aligned}
 \mu_j \text{ monotone increasing}, \quad
 \int_0^T g_j(r^*(t)) \, d\mu_j(t) = 0,\quad 
 g_j(r^*(t)) &\leq 0 \;\forall\, t \in [0,T]\\
 &\text{for all } j= 1,..,m
\end{aligned}
\end{equation*}
\end{theorem}

\begin{remark}
As a function of bounded variation, $\mu$ can be decomposed as
\begin{equation*}
 \mu = \mu_a + \mu_d + \mu_s,
\end{equation*}
where $\mu_a \in \mathrm{AC}([0,T];\R^m)$ is absolutely continuous and $\mu_d \in L^\infty(]0,T[;\R^m)$ is a step function 
covering the discontinuities of $\mu$. Moreover, $\mu_s\in C([0,T];\R^m)$ is the singular part, which is non-constant and
whose derivative vanishes almost everywhere. 
Consequently, $\dot\mu$ in \eqref{eq:rhoode2} can be replaced by $\dot\mu_a$, while \eqref{eq:rhojump2} holds also 
with $\mu_d$ instead of $\mu$.
\end{remark}

\begin{remark}
 By integration by parts one can formally derive a strong formulation of the adjoint Maxwell equations in Theorem \ref{thm:kkt}:
 \begin{align}
		-&\frac{\partial}{\partial t}\Phi(x,t) + \curl \Psi(x,t) = \varphi(x-r^*(t)) \pi(t) & &\text{in} \ \Omega \times [0,T] \label{eq:admaxa}\\
		-&\frac{\partial}{\partial t}\Psi(x,t) - \curl \Phi(x,t) = \pi(t) \times \varphi(x-r^*(t))\beta(p^*(t))& &\text{in} \ \Omega \times [0,T]\label{eq:admaxb}\\
		&\dive \Big(\frac{\partial}{\partial t} \Phi(x,t)\Big) = - \dive \big(\varphi(x-r^*(t)) \pi(t)\big)& &\text{in} \ \Omega \times [0,T]\label{eq:div1}\\
		&\dive\Big(\frac{\partial}{\partial t} \Psi(x,t)\Big)= -\dive \big(\pi(t) \times \varphi(x-r^*(t)) 
		\beta(p^*(t))\big)& &\text{in} \ \Omega \times [0,T]\label{eq:div2}\\
		&\Phi(\varsigma,t) \times n = 0, \quad \frac{\partial}{\partial t} \Psi(\varsigma,t) \cdot n = -\pi(t) \times \varphi(\varsigma-r^*(t)\beta(p^*(t)))\cdot n  & &\text{in} \ \Gamma \times [0,T]\label{eq:adbound}\\
		&\Phi(x,T)= 0, \quad \Psi(x,T) = 0 & &\text{in} \ \Omega.
	\end{align}
 Note that the right hand side in \eqref{eq:admaxa}--\eqref{eq:admaxb} does, in general, not satisfy a conservation of charge, 
 which gives rise to non-standard equations in \eqref{eq:div1} and \eqref{eq:div2} and the unusual boundary condition in \eqref{eq:adbound}.
\end{remark}

\section{Numerical investigations}\label{sec:diskret}

In the following we illustrate by means of a representative example 
that the optimal control problem \eqref{eq:optconexact} can be treated numerically. 
We  follow the analytical approach and use the reduced state system of Definition \ref{def:compstate} for our 
numerical investigations. After a brief description of the numerical method we will present some exemplary results.

\subsection{Discretization of the state system}

We start the description of the numerical method with the discretization of the state system. 
Inspired from the analytical treatment of Maxwell's equations by means of semigroup theory, we approximate 
the solution of Maxwell's equations with the help of their fundamental solution, i.e., the semigroup arising if 
$\Omega = \R^3$. We thus neglect the influence of any boundary conditions. 
In case of a single point charge, i.e., charge and current as in \eqref{eq:charge} and \eqref{eq:currentcoupl},
this fundamental solution allows an explicit representation of the arising electromagnetic fields, 
the so called Li\'enard-Wiechert fields, cf.\ e.g.\ \cite{Jackson1999, Spohn2004}:
\begin{align}
 &\begin{aligned}
  E(x, t) &= \frac{q}{4 \pi \epsilon}\,\frac{\left( 1-|\beta(p(t_{ret}))|_2^2\right)}{|R_{v}(t_{ret},p(t_{ret}))|_2^{3}}\,
  R_{v}(t_{ret},p(t_{ret})) \\
  &+ \frac{q}{4 \pi \epsilon c^2 |R_{v}(t_{ret},p(t_{ret}))|_2^3}\, R(t_{ret}) \times 
  \left(R_{v}(t_{ret},p(t_{ret}))  \times \dot {v}(p(t_{ret}))\right) \\
 \end{aligned}\label{eq:lienardE}\\
 & B(x, t) = c\, \epsilon\, \mu\, \frac{R(t_{ret})}{|R(t_{ret})|_2}  \times E(x,t)
\end{align}
with
\begin{align*}
 R(t) &:= x - r(t), \quad t_{ret} := t_{ret}(x,t) = t - \frac{R(t_{ret})}{c},\\
 R_{v}(t,p)&:=\left(R(t)-\beta(p) R(t) \right).
\end{align*}
For the numerical realization these expressions are further simplified. Firstly, we neglect the difference between 
$t$ and $t_{ret}$. Moreover, we leave out the terms arising from an acceleration of the charge, i.e., the second addend on the 
right hand side of \eqref{eq:lienardE}. In contrast to the first addend which is of order $1/R^2$, this 
term grows with $1/R$ and thus models the far field, whose influence on the movement of the particles can be neglected, 
see \cite{SpohnKomech2000}.

The Poisson equation in \eqref{eq:vweakpoisson} is discretized by means of finite elements. 
We use a uniform hexahedral mesh and piecewise trilinear and continuous 
ansatz functions for both, solution and test function, which represents a variational crime 
due to the low regularity of the very weak solution. A priori error analysis for this procedure can be found in \cite{Berggren2004}. 
The linear system of equations arising by this discretization is solved by the CG method preconditioned via an incomplete 
LU decomposition of the stiffness matrix.

Finally, the relativistic Newton-Lorentz equations \eqref{eq:nlorentzpInt}
are solved numerically by the so called Boris scheme, a second-order 
time stepping scheme especially tailored to this type of equations of motions, described in \cite{Boris1970, birdsall2004plasma}. 
It is frequently used in plasma physics
and especially for particle accelerators (as part of particle-in-cell methods), since it is an explicit and energy conserving scheme.
The physical quantities and constants involved in \eqref{eq:nlorentzpInt}  differ by several orders of magnitude, 
cf.\ Table \ref{konstanten} below. In order to avoid numerical cancellation effects, we introduce a nondimensionalization 
factor in the Newton-Lorentz equations.
In addition, cancellation also occurs in the numerical evaluation of the integrals involving $\varphi$ in \eqref{eq:newtonlorentz}. 
This is due to the small support of $\varphi$, whose diameter amounts $10^{-6}$ 
and causes larger slopes of $\varphi$ due to the normalization in \eqref{eq:smearout}. 
To circumvent these problems, we use a linear transformation to enlarge the support.
The transformed integrals are approximated by the Simpson rule weighted with $\varphi(. - r)$ 
and $\nabla \varphi(. - r)\times \beta(p)$, respectively.

\subsection{Optimization algorithm}

To keep the model physically meaningful it is of major importance to fulfill the pointwise state constraint in 
\eqref{eq:stateconst}, see Remark \ref{rem:stateconst}. This is guaranteed by a purely primal interior point approach 
in form of a $\log$-barrier method, see e.g.\ \cite[Chapter 19]{nocedal2006}. In \cite{Schiela2009, Schiela2013} this method has been proven 
to work in function space for one dimensional problems, i.e., problems involving ODEs as in our case.
The reduction of the homotopy parameter associated with the primal interior point method follows 
an update strategy by \cite[Section 19.3]{nocedal2006}.

For the optimization algorithm we reduce the optimal control problem to an optimization problem in the control variable 
$u$ only, which is justified by Theorem \ref{thm:stateex}. The major advantage of this procedure is a significant 
reduction of the number of optimization variables, since the control $u$ is only one dimensional. It does not depend on time, 
and has its support on $\Gamma$ instead of the whole domain $\Omega$. The dimension of the optimization problem reduced to 
the control variable, thus, amounts to the number of nodes on the boundary only. This allows to employ optimization methods, which 
require large memory demand like the BFGS method, see e.g. \cite[Section 6.1]{nocedal2006}. Thanks to the reduction of the dimension the 
BFGS method can be run for a moderate number of degrees of freedom on a computer with 4GB RAM without any limited memory modification.
In order to globalize the method, we perform a curvature test to switch from the BFGS direction to the negative gradient
of the objective, if necessary, and apply a line-search according to the Armijo rule.

As a consequence of this reduction approach the mapping $u \mapsto \JJ(r(u),u)$ as well as its derivative have to be 
evaluated in every iteration of the optimization algorithm. Here $r(u)$ denotes the $r$-component of the
solution of state system associated with $u$.  
The derivative of $u \mapsto \JJ(r(u),u)$ is computed numerically by means of the automatic differentiation tool ADiMat \cite{Bischof2002}. 
As the number of control variables is much higher than the number of output variables, which is just a real number, 
we use the reverse mode. Moreover, we exclude the linear parts of the solution mapping of the state system from automatic 
differentiation to differentiate them by hand. This especially concerns the iterative solver of Poisson's equation.
To summarize we thus follow a first-discretize-then optimize approach. It is not clear whether the discrete adjoint 
equation arising in this way can be interpreted as a suitable discretization of the adjoint system in Theorem \ref{thm:kkt}. 
In particular, the adjoint Boris scheme gives rise to future research with regard to its stability and consistency.

\subsection{Test setting}

For the numerical realization we chose an electron as particle. The mass at rest and the charge are chosen appropriately, see 
Table \ref{konstanten}. 
The computational domain $\Omega$ is a cube of size length $2\cdot 10^{-3}$~m. 
For the subdomain $\tilde\Omega$ arising in the state constraint \eqref{eq:stateconst} we chose an inner cube of size length 
$2\cdot 10^{-4}$~m.
As the electron is almost moving with the speed of light, the end time was set to $T = 2 \cdot10^{-10}$~s.

\begin{table}[h!]
	\begin{center}
		\begin{tabular}{lcc}
			\toprule
			Quantity & Symbol & Value (in SI units)\\
			\midrule
			speed of light (in vacuum) & $c$ &  $2. 9979 \cdot 10^8 \  \textup{m s}^{-1}$\\
			permittivity of free space & $\epsilon$ &  $8.8541 \cdot  10^{-12}\ \textup{F m}^{-1}$\\
			permeability of free space & $\mu$ & $4 \pi \cdot 10^{-7}\  \textup{H m}^{-1}$\\
			electron rest mass & $m_0^q$ & $9.1093 \cdot 10^{-31}\ \textup{kg}$\\
			electric charge & $q$ & $1.6021 \cdot 10^{-19}\ \textup{C}$\\
			\bottomrule
		\end{tabular}
		\caption{Physical constants.}\label{konstanten}
	\end{center}
\end{table}

For the numerical computations we focus on optimizing the particle position at end time, i.e., we choose 
\begin{equation*}
 J_1(r) \equiv 0, \quad J_2(r) = \frac{1}{2}\, |r - r_d|_2^2
\end{equation*}
for the contributions to the objective in \eqref{eq:tildeP} and \eqref{eq:optconexact}, respectively.
Furthermore, the Tikhonov parameter $\alpha$ in the objective is set to $\alpha = 10^{-9}$ 
to compensate for the comparatively large values of the control. 
Consequently we are mainly interested in steering the particle beam at a given end time to a fixed position $r_d$. 
As a stopping criterion for the overall algorithm we check if the relative error 
between the desired particle position $r_d$ and the computed one is below a given tolerance.

For the computations presented in the following section, we used an equidistant mesh with 17,576 nodes. 
This amounts to 7,504 nodes on the boundary, i.e., the number of unknown control variables, which corresponds to the dimension of the 
optimization problem. For the numerical integration of the ODE we used an equidistant time step of $10^{-12}$~s.

\subsection{Numerical results}\label{sec:numerics}

The particle trajectories for selected iterations of the optimization 
algorithms are shown in Figures \ref{traj1} to \ref{traj47}. 
While the particle is colored in black, we marked the desired end position in the upper left corner in grey.
It is to be noted that the control $u$ only influences the magnetic field, 
which in turn cannot slow down or accelerate this particle beam since its contribution to the Lorentz force only acts perpendicular 
to the direction of motion, cf.\ \eqref{eq:nlorentza}. This causes spiral shaped trajectories such as the ones depicted in figures \ref{traj0} to \ref{traj47}.
The desired end position has been reached after 47 iterations of the optimization algorithm with an accuracy of $3.5\cdot 10^{-8}$~m.

\begin{figure}[h]
 	\centering
  	\begin{minipage}[b]{6cm}
		\includegraphics[width=6cm]{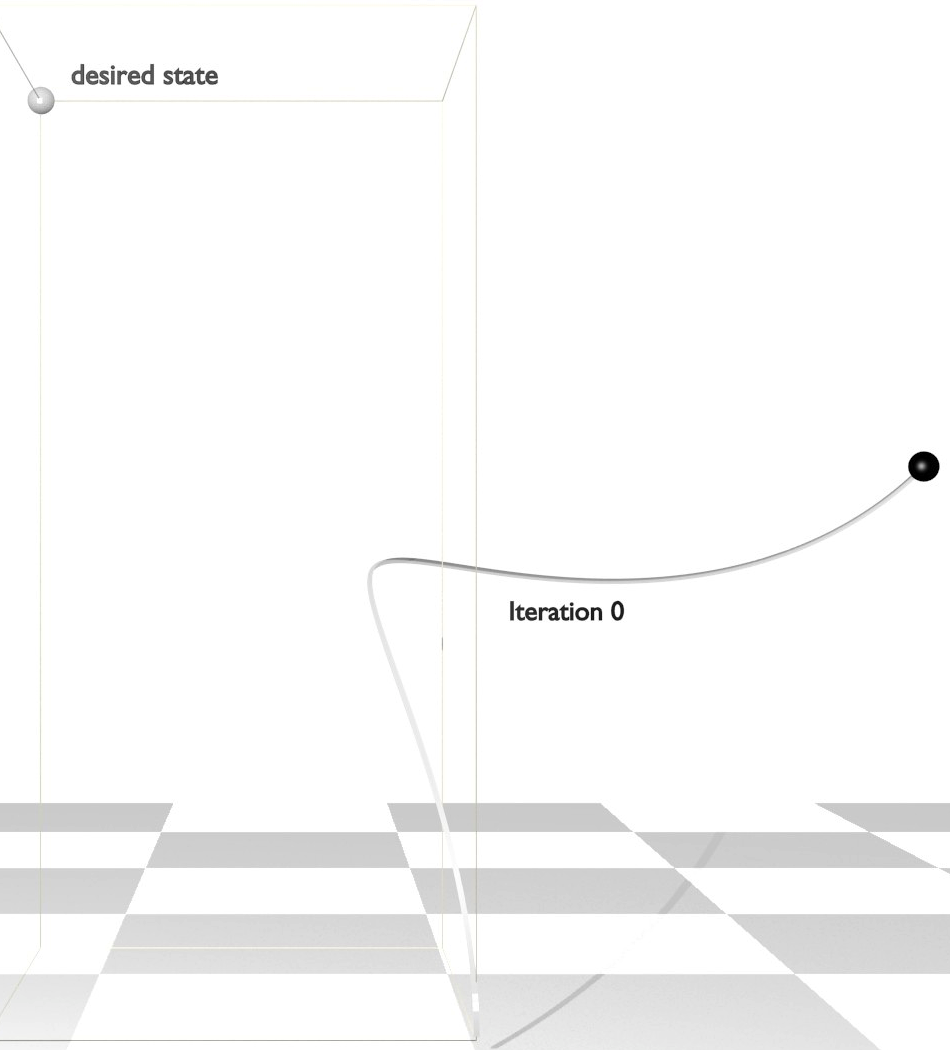}
		\captionsetup{font=small} 
    		\captionsetup[figure]{labelfont=small} 
    		\caption{Particle trajectory in iteration 0.}
		\label{traj0}
	\end{minipage}
	\begin{minipage}[b]{6cm}
		\includegraphics[width=6cm]{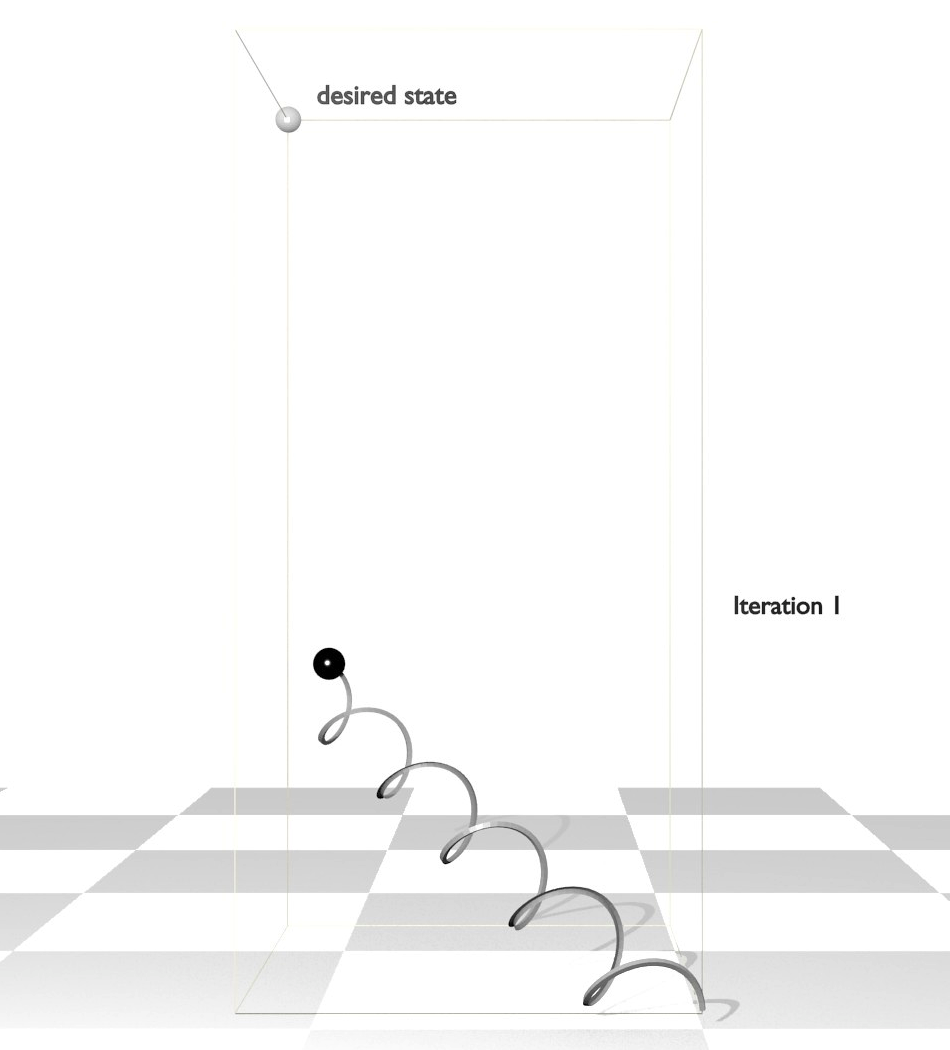}
		\captionsetup{font=small} 
    		\captionsetup[figure]{labelfont=small} 
    		\caption{Particle trajectory in iteration 1.}
		\label{traj1}
	\end{minipage}
\end{figure}
\begin{figure}[h]
 	\centering
  	\begin{minipage}[b]{6cm}
		\includegraphics[width=6cm]{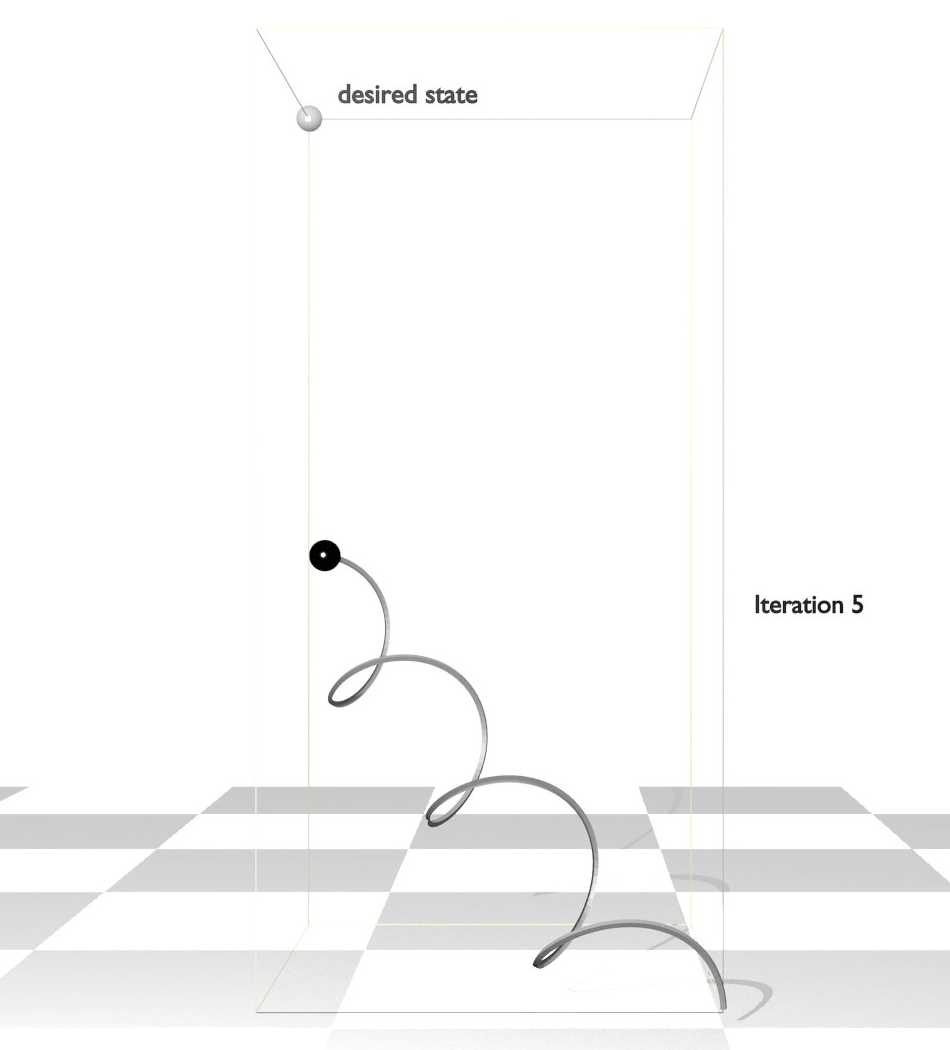}
		\captionsetup{font=small} 
    		\captionsetup[figure]{labelfont=small} 
    		\caption{Particle trajectory in iteration 5.}
		\label{traj5}
	\end{minipage}
	\begin{minipage}[b]{6cm}
		\includegraphics[width=6cm]{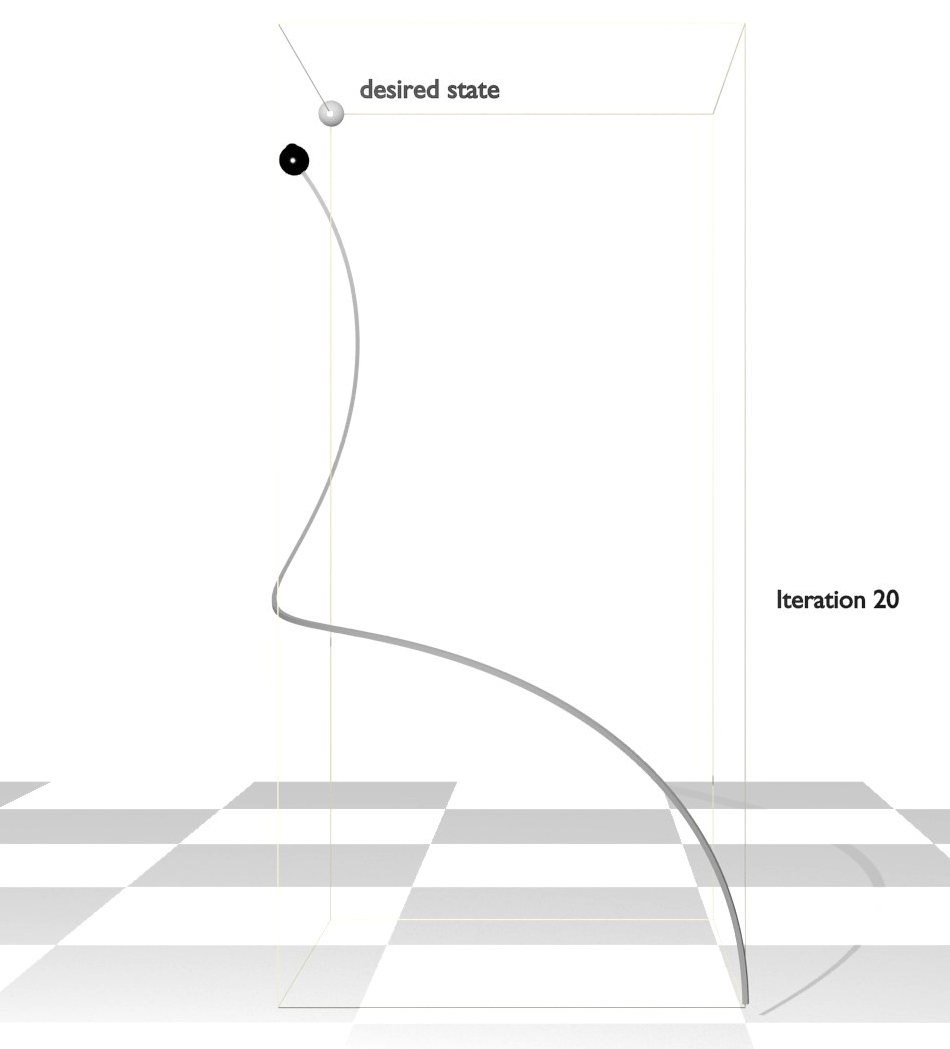}
		\captionsetup{font=small} 
    		\captionsetup[figure]{labelfont=small} 
    		\caption{Particle trajectory in iteration 20.}
		\label{traj20}
	\end{minipage}
\end{figure}
\begin{figure}[h]
 	\centering
  	\begin{minipage}[b]{6cm}
		\includegraphics[width=6cm]{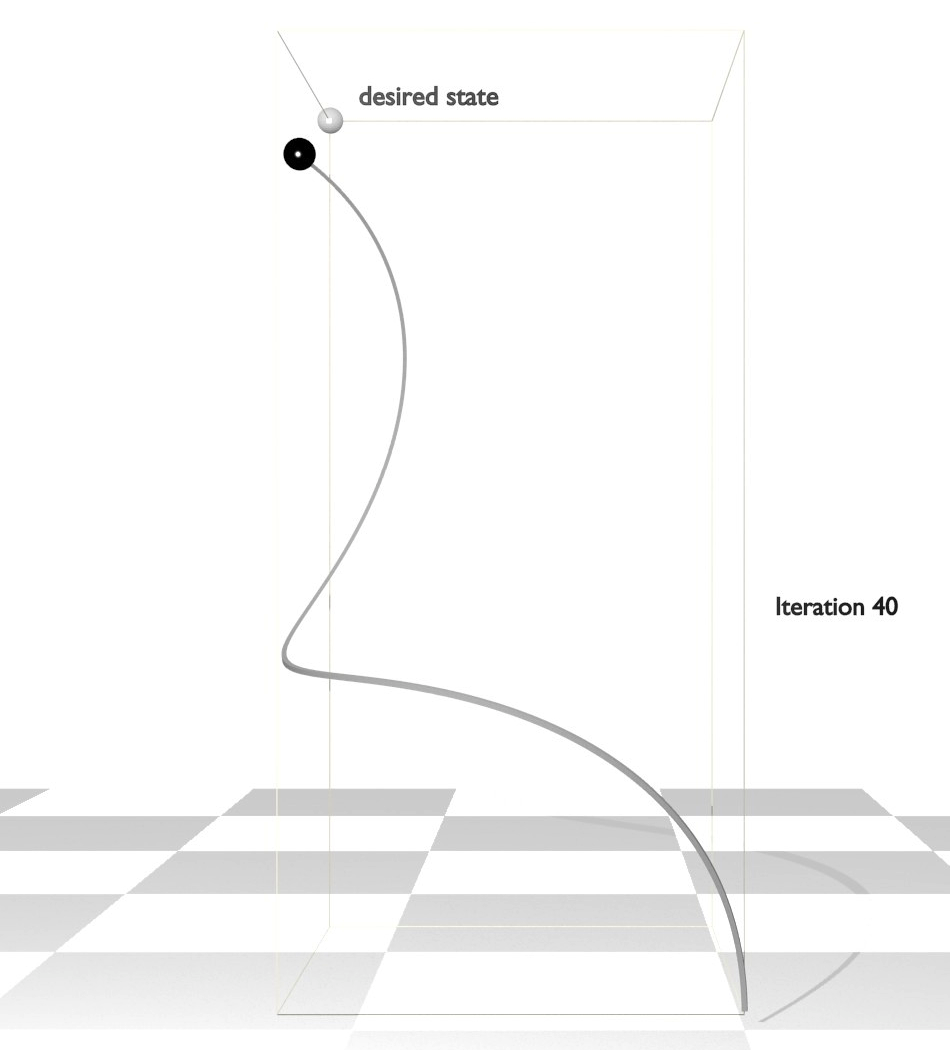}
		\captionsetup{font=small} 
    		\captionsetup[figure]{labelfont=small} 
    		\caption{Particle trajectory in iteration 40.}
		\label{traj40}
	\end{minipage}
	 \begin{minipage}[b]{6cm}
		\includegraphics[width=6cm]{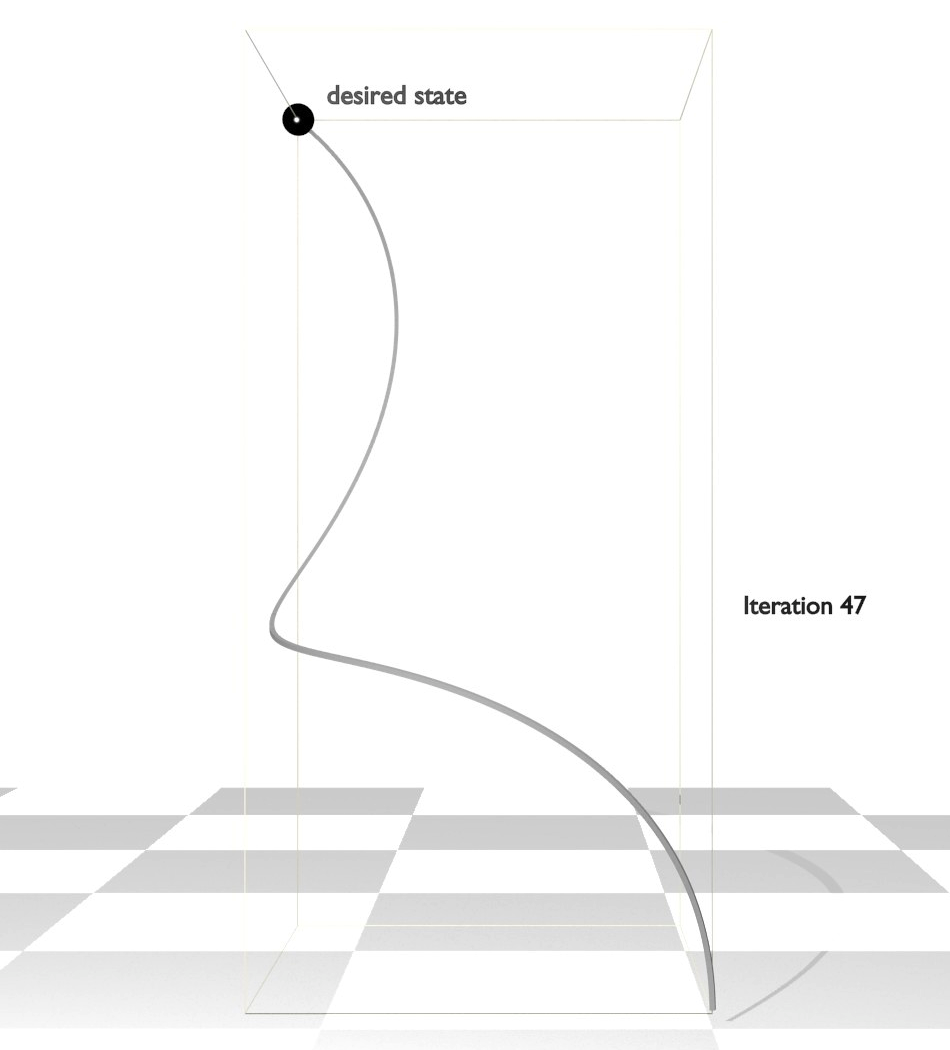}
		\captionsetup{font=small} 
    		\captionsetup[figure]{labelfont=small} 
    		\caption{Particle trajectory in iteration 47.}
		\label{traj47}
	\end{minipage}
\end{figure}


The optimal external magnetic field on the boundary of $\Omega$ generated by the optimal control $u^*$ is shown in Figures \ref{maga} and \ref{magb}. 
\begin{figure}[h]
  \centering
  \begin{minipage}[b]{6cm}
    \includegraphics[width=6cm]{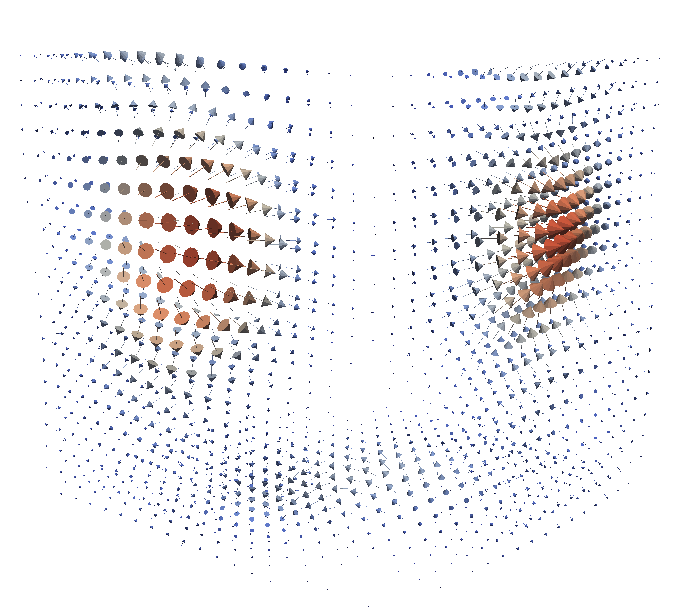}
    \captionsetup{font=small} 
    \captionsetup[figure]{labelfont=small} 
    \caption{Front view of  external magnetic field.}
    \label{maga}
  \end{minipage}
  \begin{minipage}[b]{6cm}
    \includegraphics[width=6cm]{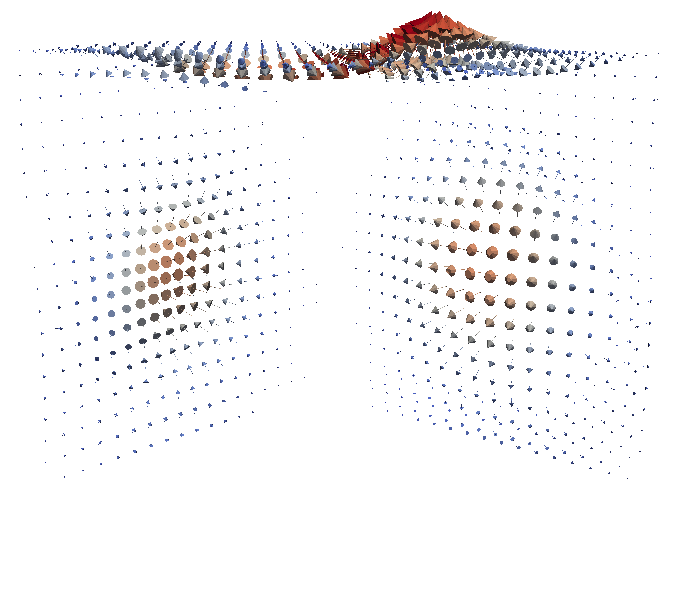} 
    \captionsetup{font=small} 
    \captionsetup[figure]{labelfont=small}  
    \caption{Back view of external magnetic field.}
    \label{magb}
  \end{minipage}
\end{figure}



Table \ref{tabtest1a} shows the convergence history of the globalized BFGS interior point method. 
Beside the objective value and Euclidean norm of the gradient, Table \ref{tabtest1a} shows the used descent direction for selected iterations, 
where ``BFGS'' refers to the BFGS direction and ``Grad'' is the negative gradient.
\begin{table}[h]
\begin{center}
\begin{tabular}{lrrr|lrrr}
\toprule
Iteration & f-value & gradient & step & Iteration & f-value & gradient & step\\
\midrule
0 & 0.3835 & - & - & 35 & 0.0060& 1.7E-5& BFGS\\
1 & 0.1509 & 0.0024& Grad & 40 & 1.6E-5& 1.7E-5& Grad  \\ 
5 & 0.0748& 6.0E-4& BFGS & 42 &9.9E-8& 6.0E-6& BFGS  \\
10 & 0.0091& 2.6E-4& BFGS  & 44& 4.0E-8& 2.7E-7& BFGS  \\
20 & 0.0086& 2.3E-4& Grad  & 46 & 2.5E-9& 1.4E-8& BFGS  \\
25 & 0.0064& 6.0E-5& BFGS  & 47 & 6.1E-10& 5.4E-9& BFGS  \\
30 &0.0061& 3.9E-5& BFGS  \\
\bottomrule
\end{tabular}
\captionsetup{font=small} 
 \captionsetup[figure]{labelfont=small}  
\caption{Convergence history of the optimization algorithm.}\label{tabtest1a}
\end{center}
\end{table}

\begin{appendix}

\section{Proof of Theorem \ref{thm:stateex}}\label{app:stateex}

 Clearly, $y$ solves \eqref{eq:nlorentzsred} if and only if it is a fixed point of 
 \begin{equation*}
  G: C([0,T];\R^3)^2 \to C([0,T];\R^3)^2,\quad 
  G(y)(t):= \int_0^t f(y)(\tau) d\tau.
 \end{equation*}  
 We show that $G$ is contractive, if we equip the set of continuous functions with the following equivalent norm
 \begin{equation*}
  \|y\|_{G}:= \max_{t\in [0,T]} e^{-Lt} |y(t)|_2.
 \end{equation*}
 To this end, observe that for every $v\in C([0,t];\R^3)^2$ and every $\tau\in [0,T]$ there holds
 \begin{equation*}
  \|v\|_{C([0,\tau];\R^3)^2} \leq \max_{s\in [0,\tau]} e^{Ls}\,\max_{s\in [0,\tau]} e^{-Ls} |v(s)|_2 \leq e^{L\tau}\,\|v\|_G.
 \end{equation*}
 Then we obtain by means of Lemma \ref{lm:Lipred}
 \begin{equation*}
 \begin{aligned}
  \|G(y) - G(v)\|_G 
  &\leq \max_{t\in [0,T]} e^{-Lt} \int_0^t |f(y)(\tau) - f(v)(\tau)| d\tau\\
  &\leq L \max_{t\in [0,T]} e^{-Lt} \int_0^t \|y - v\|_{C([0,\tau];\R^3)^2} d\tau\\
  &\leq L\,\max_{t\in [0,T]} e^{-Lt} \int_0^t e^{L\tau}\,d\tau\,\|y - v\|_G \leq \big(1- e^{-LT}\big)\|y - v\|_G,
 \end{aligned}
 \end{equation*}
 i.e., the desired contractivity of $G$. Thus Banach's fixed point theorem gives the existence of a unique solution to 
 \eqref{eq:nlorentzsred} as claimed.

 To prove the a priori estimate we again abbreviate $(r,p) := y + (r_0, p_0)$. Then \eqref{eq:sigmaest} implies
 for an arbitrary $t\in [0,T]$ that
 \begin{equation*}
  |f_2(y)(t)| \leq c.
 \end{equation*}  
 Beside \eqref{eq:phiest}, the conditions on $\varphi$ in \eqref{eq:smearout} clearly give that for every $r\in \R^3$
 \begin{equation*}
  \|\nabla \varphi(x-r)\|_X \leq \sqrt{|\Omega|}\max_{x\in \R^3} |\nabla\varphi(x)| < \infty,
  \quad \|\varphi(x-r)\|_{L^2(\Gamma)} \leq C_\varphi\, \sqrt{|\Gamma|} < \infty.
 \end{equation*}  
 Thus, \eqref{eq:semiexp}, \eqref{eq:sigmaest}, \eqref{eq:phiest}, and the definition of $\beta$ in \eqref{eq:sigmac} give
 \begin{equation*}
 \begin{aligned}
  |f_1(y)(t)| &\leq q \Big(\|\varphi(.-r(t))\|_{L^2(\Omega)} \|F_L(r,p)(t)\|_X \\
  &\qquad + \|(-\Delta^*)^{-1}R\|_{\LL(L^2(\Gamma),L^2(\Omega))}\|u\|_{L^2(\Gamma)} \|\nabla \varphi(.-r(t)\|_{L^2(\Omega)} 
  |\beta(p(t))|_2\\
  &\qquad + \|u\|_{L^2(\Gamma)} \|\varphi(.-r(t))\|_{L^2(\Gamma)} |\beta(p(t))|_2\Big)\\
  &\leq C_1 \,\|u\|_{L^2(\Gamma)} + C\, \|F_L(r,p)(t)\|_X.
 \end{aligned}
 \end{equation*}
 In view of $|\beta(p)|_2 \leq 1$ for all $p\in \R^3$, cf.\ again \eqref{eq:sigmac} and \eqref{eq:sigmaest}, $F_L$ can be estimated by
 \begin{equation*}
  \|F_L(r,p)(t)\|_X \leq 2 M e^{\omega T} \Big( \|(E_0, B_0)\|_{X\times X} + \|j(r,p)\|_{L^1([0,T];X)}\Big)
 \end{equation*}
 with 
 \begin{equation*}
  \|j(r,p)\|_{L^1([0,T];X)} \leq q\,T\,\sqrt{C_\varphi}\,c,
 \end{equation*}
 see \eqref{eq:jest}. Therefore, we arrive at
 \begin{equation*}
  |\dot y(t)| = |f(y)(t)| \leq C_1 \,\|u\|_{L^2(\Gamma)} + C_2 \quad\forall\, t\in [0,T]
 \end{equation*}
 with constants $C_1, C_2 > 0$ independent of $t$, $u$, and $y$. As $y(0) = 0$, this gives the desired estimate.

\section{Proof of Lemma \ref{lem:nbv}}\label{app:nbv}

By the Riesz representation theorem there exists a unique function $\mu \in \mathrm{NBV}([0,T];\R^m)$ such that
\begin{multline}\label{eq:rieszbv}
 \dual{\lambda}{g'(r^*)\phi_r}_{C([0,T];\R^m)^*, C([0,T];\R^m)} 
 = \int_0^T \big(g'(r^*(t))\phi_r(t)\big) \cdot d\mu(t) \\ \forall \,\phi_r\in C([0,T];\R^3).
\end{multline}
Moreover, $\lambda \geq 0$ implies that $\mu$ is monotonically increasing as claimed. 
Taking the definition of $\JJ$ into account we find
\begin{equation*}
 \sdual{\ddp{\JJ}{r}(r^*, u^*)}{\phi_r}_{Y^*,Y}
 = \int_0^T J_1'(r^*(t))\phi_r(t)\,dt + J_2'(r^*(T))\phi_r(T).
\end{equation*}
By inserting this together with \eqref{eq:rieszbv} in \eqref{eq:admitmass} we arrive at
\begin{multline*}
 \int_0^T \dot\phi(t)\cdot \omega(t)\,dt 
 + \int_0^T \phi(t) \cdot \Bigg[A(y^*,u^*,\omega)(t) - \vektor{\nabla J_1(r^*(t))}{0}\Bigg] dt\\
 - \int_0^T \big(g'(r^*(t))\phi_r(t)\big) \cdot d\mu(t) = 0
 \quad\forall\, \phi \in C^{\infty}_0([0,T];\R^3)^2.
\end{multline*} 
In view of \eqref{eq:Adef},  the continuity of $B^*$, $E^*$, and $y^*$ w.r.t.\ time and $\omega \in L^2(]0,T[;\R^3)^2$
implies $A(y^*, u^*, \omega) \in L^2(]0,T[;\R^3)^2$.
Thus, according to the Du Bois Raymond theorem for Stieltjes integrals, see e.g. \cite[Lemma 3.1.9]{Gerdts2012},
the equivalence class $\omega$ admits a representation as $\mathrm{BV}$-function, denoted by the same symbol 
for simplicity, which fulfills for all $t\in [0,T]$
\begin{align}
 \varrho(t)
 &= 
 \begin{aligned}[t]
  \varrho(T) &- \int_t^T \big[A_r(y^*, u^*, \omega)(\tau) - \nabla J_1(r^*(\tau))\big] d\tau\\
  &+ \int_t^T g'(r^*(\tau))^\top d\mu(\tau)
 \end{aligned}  \label{eq:rho}\\
 \pi(t) &= \pi(T) -\int_t^T A_p(y^*, u^*, \omega)(\tau)\,d\tau.\label{eq:pi}
\end{align}
The later equation immediately implies \eqref{eq:piode} and, since 
$\pi, \varrho \in \mathrm{BV}([0,T];\R^3) \embed L^\infty(]0,T[;\R^3)$, this ODE gives the desired regularity of $\pi$.

As a function of bounded variation $\mu$ has at most countably many discontinuities and is differentiable almost everywhere in $]0,T[$. 
Moreover, since $\mu$ is in addition monotonically increasing, there holds
\begin{equation*}
 \frac{d}{dt} \int_t^T g'(r^*(\tau))^\top d\mu(\tau) = -g'(r^*(t))^\top \dot\mu(t) \quad \text{f.a.a.\ } t \in ]0,T[
\end{equation*}
see e.g.\ \cite[Lemma 2.1.26]{Gerdts2012}. Thus \eqref{eq:rho} gives \eqref{eq:rhoode}.

Integrating the last integral in \eqref{eq:rho} by parts leads to
\begin{equation*}
\begin{aligned}
 &\varrho(t) + g'(r^*(t))\cdot \mu(t)\\
 &= \varrho(T) + g'(r^*(T)) \cdot \mu(T)\\
 &\quad - \int_t^T \Big[A_r(y^*, u^*, \omega)(\tau) - \nabla J_1(r^*(\tau))+ \sum_{j=1}^m \mu_j(\tau) \, g_j''(r^*(\tau))\dot r^*(\tau)\Big] d\tau.
\end{aligned}
\end{equation*}
As the right hand side is continuous, the discontinuities of $\varrho$ are therefore located at the same points as the ones of $\mu$. 
Moreover, as $\mu$ is of bounded variation, one has
\begin{equation*}
 \int_t^T g'(r^*(\tau))^\top d\mu(\tau) - \lim_{\varepsilon \searrow 0} \int_{t-\varepsilon}^T g'(r^*(\tau))^\top d\mu(\tau) 
 = g'(r^*(t))^\top \big(\lim_{\varepsilon \searrow 0} \mu(t-\varepsilon) - \mu(t)\big)
\end{equation*}
for every $t\in ]0,T]$, cf.\ e.g.\ \cite[p.\ 66]{Gerdts2012}. Since $A_r(y^*, u^*, \omega)(.) - \nabla J_1(r^*(.)) \in L^2(]0,T[;\R^3)$,  
\eqref{eq:rho} therefore implies \eqref{eq:rhojump}.

Integrating the first integral in \eqref{eq:admitmass} by parts yields
\begin{equation}\label{eq:endwert}
\begin{aligned}
 &- \int_0^T \phi(t)\cdot d\omega(t)
 + \int_0^T \phi(t) \cdot \Bigg[A(y^*,u^*,\omega)(t) - \vektor{\nabla J_1(r^*(t))}{0}\Bigg] dt\\
 & - \int_0^T \big(g'(r^*(t))\phi_r(t)\big) \cdot d\mu(t)
 =  J_2'(r^*(T))\phi_r(T)- \phi(T)\cdot\omega(T) \quad \forall\,\phi\in Y.
\end{aligned}
\end{equation}
The continuity of $g'(r^*(\,.\,))$ gives that 
\begin{equation*}
 \nu(t) = \int_t^T g'(r^*(\tau))^\top d\mu(\tau)
\end{equation*}
is of bounded variation. Since $\phi_r$ also continuous, we arrive at 
\begin{equation*}
 \int_0^T \big(g'(r^*(t))\phi_r(t)\big) \cdot d\mu(t)
 = -\int_0^T \phi_r(t)^\top d\nu(t),
\end{equation*}
cf.\ e.g.\ \cite[p.\ 67]{Gerdts2012}. Hence, thanks to \eqref{eq:rho} and \eqref{eq:pi}, 
\eqref{eq:endwert} gives $\omega(T)\cdot \phi(T)  = J_2'(r^*(T))\phi_r(T)$ for all $\phi \in Y$, which in turn 
yields the desired final time conditions in \eqref{eq:piend} and \eqref{eq:rhoend}.

\end{appendix}

 \bibliographystyle{plain}

\begin{thebibliography}{10}

\bibitem{Schneidmiller:2007tx}
W~Ackermann and et~al.
\newblock {Operation of a free-electron laser from the extreme ultraviolet to
  the water window}.
\newblock {\em Nature Photonics}, 1(6):336--342, June 2007.

\bibitem{Alonso1996}
A.~Alonso and A.~Valli.
\newblock Some remarks on the characterization of the space of tangential
  traces of $\boldsymbol{H}(\rot; \, \omega)$ and the construction of an
  extension operator.
\newblock {\em manuscripta mathematica}, 89(1):159--178, 1996.

\bibitem{BauerDeckertDuerr2013}
G.~Bauer, D.-A. Deckert, and D.~D\"urr.
\newblock {M}axwell-{L}orentz dynamics of rigid charges.
\newblock {\em Communications in Partial Differential Equations},
  38(9):1519--1538, 2013.

\bibitem{Berggren2004}
M.~Berggren.
\newblock Approximations of very weak solutions to boundary-value problems.
\newblock {\em Siam J. Numer. Anal.}, 42(2):860--877, 2004.

\bibitem{birdsall2004plasma}
C.K. Birdsall and A.B. Langdon.
\newblock {\em {P}lasma Physics via Computer Simulation}.
\newblock Series in Plasma Physics. Taylor \& Francis, 2004.

\bibitem{Bischof2002}
C.H. Bischof, H.M. B{\"u}cker, B.~Lang, A.~Rasch, and A.~Vehreschild.
\newblock Combining source transformation and operator overloading techniques
  to compute derivatives for {MATLAB} programs.
\newblock In {\em Proceedings of the Second {IEEE} International Workshop on
  Source Code Analysis and Manipulation ({SCAM} 2002)}, pages 65--72, Los
  Alamitos, CA, USA, 2002. IEEE Computer Society.

\bibitem{Boris1970}
J.P. Boris.
\newblock Relativistic plasma simulation - optimization of a hybrid code.
\newblock {\em In Proc. Fourth Conf. Numerical Simulation of Plasmas}, pages
  3--67, 1970.

\bibitem{cartan1983}
H.P. Cartan.
\newblock {\em Differential {C}alculus}.
\newblock Hermann, 1983.

\bibitem{Casas1986}
E.~Casas.
\newblock Control of an elliptic problem with pointwise state constraints.
\newblock {\em SIAM Journal on Control and Optimization}, 24(6):1309--1318,
  1986.

\bibitem{Casas1993}
E.~Casas.
\newblock Boundary control of semilinear elliptic equations with pointwise
  state constraints.
\newblock {\em SIAM Journal on Control and Optimization}, 31(4):993--1006,
  1993.

\bibitem{CasasRay2006}
E.~Casas and J.~Raymond.
\newblock Error estimates for the numerical approximation of {D}irichlet
  boundary control for semilinear elliptic equations.
\newblock {\em SIAM Journal on Control and Optimization}, 45(5):1586--1611,
  2006.

\bibitem{Dauge1988}
M.~Dauge.
\newblock {\em Elliptic boundary value problems on corner domains : smoothness
  and asymptotics of solutions}.
\newblock Lecture notes in mathematics. Springer, Berlin, 1988.

\bibitem{Dautray1992:1}
R.~Dautray and J.-L. Lions.
\newblock {\em Mathematical Analysis and Numerical Methods for science and
  technology}, volume~5.
\newblock Spinger, Berlin, 1992.

\bibitem{Dautray1992:2}
R.~Dautray and J.-L. Lions.
\newblock {\em Mathematical Analysis and Numerical Methods for science and
  technology}, volume~3.
\newblock Spinger, Berlin, 1992.

\bibitem{DeckelnickGuntherHinze2009}
K.~Deckelnick, A.~G\"unther, and M.~Hinze.
\newblock Finite element approximation of {D}irichlet boundary control for
  elliptic pdes on two- and three-dimensional curved domains.
\newblock {\em SIAM Journal on Control and Optimization}, 48(4):2798--2819,
  2009.

\bibitem{DruetKleinSprekelsYouseptTroeltzschYousept2011}
P.~Druet, O.~Klein, J.~Sprekels, F.~Tr\"oltzsch, and I.~Yousept.
\newblock Optimal control of three-dimensional state-constrained induction
  heating problems with nonlocal radiation effects.
\newblock {\em SIAM Journal on Control and Optimization}, 49(4):1707--1736,
  2011.

\bibitem{Falconi2014}
M.~Falconi.
\newblock Global solution of the electromagnetic field-particle system of
  equations.
\newblock {\em Journal of Mathematical Physics}, 55(10):1--12, 2014.

\bibitem{Geddes:2004bi}
CGR Geddes, C~Toth, J~van Tilborg, E~Esarey, C~B Schroeder, D~Bruhwiler,
  C~Nieter, J~Cary, and W~P Leemans.
\newblock {High-quality electron beams from a laser wakefield accelerator using
  plasma-channel guiding}.
\newblock {\em Nature}, 431(7008):538--541, 2004.

\bibitem{Gerdts2012}
M.~Gerdts.
\newblock {\em Optimal control of {ODE}s and {DAE}s}.
\newblock De Gruyter, Berlin, 2012.

\bibitem{GiraultRaviart1986}
V.~Girault and P.-A. Raviart.
\newblock {\em Finite Element Methods for {N}avier-{S}tokes Equations}.
\newblock Springer, Berlin, 1986.

\bibitem{Gjonaj2006}
E.~Gjonaj, T.~Lau, S.~Schnepp, F.~Wolfheimer, and T.~Weiland.
\newblock Accurate modelling of charged particle beams in linear accelerators.
\newblock {\em New Journal of Physics}, 8:1--21, August 2006.

\bibitem{GriesseKunisch2006}
R.~Griesse and K.~Kunisch.
\newblock Optimal control for a stationary {MHD} system in velocity--current
  formulation.
\newblock {\em SIAM Journal on Control and Optimization}, 45(5):1822--1845,
  2006.

\bibitem{HartlSethiVickson1995}
R.~F. Hartl, S.~P. Sethi, and R.~G. Vickson.
\newblock A survey of the maximum principles for optimal control problems with
  state constraints.
\newblock {\em SIAM Review}, 37(2):181--218, 1995.

\bibitem{HK06}
M.~Hinterm{\"u}ller and K.~Kunisch.
\newblock Feasible and non-interior path-following in constrained minimization
  with low multiplier regularity.
\newblock {\em SIAM J. Control and Optim.}, 45:1198--1221, 2006.

\bibitem{HinzePinnauUlbrichUlbrich09}
M.~Hinze, R.~Pinnau, M.~Ulbrich, and S.~Ulbrich.
\newblock {\em Optimization with {PDE} constraints}, volume~23 of {\em
  Mathematical Modelling: Theory and Applications}.
\newblock Springer, New York, 2009.

\bibitem{ImaikinKomechMauser2004}
V.~Imaikin, A.~Komech, and N.~Mauser.
\newblock Soliton-type asymptotics for the coupled {M}axwell-{L}orentz
  equations.
\newblock {\em Annales Henri Poincar\'e}, 5:1117--1135, 2004.

\bibitem{Jackson1999}
J.~D. Jackson.
\newblock {\em Classical electrodynamics}.
\newblock Wiley, New York, 3rd edition, 1999.

\bibitem{SpohnKomech2000}
A.~Komech and H.~Spohn.
\newblock Long-time asymptotics for the coupled {M}axwell-{L}orentz equations.
\newblock {\em Communications in Partial Differential Equations}, 25:559--584,
  2000.

\bibitem{KunischVexler2007}
K.~Kunisch and B.~Vexler.
\newblock Constrained {D}irichlet boundary control in $\boldsymbol{L}^2$ for a
  class of evolution equations.
\newblock {\em SIAM Journal on Control and Optimization}, 46(5):1726--1753,
  2007.

\bibitem{Li:dn}
Z~Li, V~Akcelik, A~Candel, S~Chen, L~Ge, A~Kabel, L~Q Lee, C~Ng, E~Prudencio,
  G~Schussman, R~Uplenchwar, L~Xiao, and K~Ko.
\newblock {Towards simulation of electromagnetics and beam physics at the
  petascale}.
\newblock In {\em 2007 IEEE Particle Accelerator Conference (PAC)}, pages
  889--893. IEEE, 2007.

\bibitem{MayRannacherVexler2013}
S.~May, R.~Rannacher, and B.~Vexler.
\newblock Error analysis for a finite element approximation of elliptic
  {D}irichlet boundary control problems.
\newblock {\em SIAM Journal on Control and Optimization}, 51(3):2585--2611,
  2013.

\bibitem{MRT06}
C.~Meyer, A.~R{\"o}sch, and F.~Tr{\"o}ltzsch.
\newblock Optimal control of pdes with regularized pointwise state constraints.
\newblock {\em Comp. Optim. Appl.}, 33:209--228, 2006.

\bibitem{NicaiseStingelinTroeltzsch2014}
S.~Nicaise, S.~Stingelin, and F.~Tr\"oltzsch.
\newblock On two optimal control problems for magnetic fields.
\newblock {\em Computational Methods in Applied Mathematics}, 14:555--573,
  2014.

\bibitem{NicaiseStingelinTroeltzsch2015}
S.~Nicaise, S.~Stingelin, and F.~Tr\"oltzsch.
\newblock Optimal control of magnetic fields in flow measurement.
\newblock {\em Discrete and Continuous Dynamical Systems-S}, 8:579--605, 2015.

\bibitem{NicaiseTroeltzsch2014}
S.~Nicaise and F.~Tr\"oltzsch.
\newblock A coupled {M}axwell integrodifferential model for magnetization
  processes.
\newblock {\em Mathematische Nachrichten}, 287(4):432--452, 2014.

\bibitem{nocedal2006}
J.~Nocedal and S.~Wright.
\newblock {\em Numerical Optimization}.
\newblock Springer Series in Operations Research and Financial Engineering.
  Springer, 2006.

\bibitem{OFPhanSteinbach2010}
G.~Of, T.~X. Phan, and O.~Steinbach.
\newblock Boundary element methods for {D}irichlet boundary control problems.
\newblock {\em Mathematical Methods in the Applied Sciences},
  33(18):2187--2205, 2010.

\bibitem{Pazy1983}
A.~Pazy.
\newblock {\em Semigroups of linear operators and applications to partial
  differential equations}, volume~44.
\newblock Springer, New York, 1st edition, 1983.

\bibitem{Rossbach1992}
J.~Rossbach and P.~Schm{\"u}ser.
\newblock Basic course on accelerator optics.
\newblock In S.~Turner, editor, {\em CAS - CERN Accelerator School: 5th General
  accelerator physics course}, volume~I, 1992.

\bibitem{Schiela2009}
A.~Schiela.
\newblock Barrier methods for optimal control problems with state constraints.
\newblock {\em Siam J. Optim.}, 20(2):1002--1031, 2009.

\bibitem{Schiela2013}
A.~Schiela.
\newblock An interior point method in function space for the efficient solution
  of state constrained optimal control problems.
\newblock {\em Mathematical Programming}, 138:83--114, 2013.

\bibitem{Spohn2004}
H.~Spohn.
\newblock {\em Dynamics of charged particles and their radiation field}.
\newblock Cambridge University Press, Cambridge, 2004.

\bibitem{Temam1977}
R.~Temam.
\newblock {\em {N}avier-{S}tokes equations: theory and numerical analysis}.
\newblock North-Holland, Amsterdam, 1977.

\bibitem{YouseptTroeltzsch2012}
F.~Tr\"oltzsch and I.~Yousept.
\newblock {PDE}-constrained optimization of time-dependent 3{D} electromagnetic
  induction heating by alternating voltages.
\newblock {\em ESAIM: Mathematical Modelling and Numerical Analysis},
  46:709--729, 7 2012.

\bibitem{Wille2000}
K.~Wille.
\newblock {\em The Physics of Particle Accelerators}.
\newblock Oxford University Press, 2000.

\bibitem{Yousept2010}
I.~Yousept.
\newblock Optimal control of a nonlinear coupled electromagnetic induction
  heating system with pointwise state constraints.
\newblock {\em Ann. Acad. Rom. Sci. Ser. Math. Appl.}, 2(1):45--77, 2010.

\bibitem{Yousept2012:1}
I.~Yousept.
\newblock Finite element analysis of an optimal control problem in the
  coefficients of time-harmonic eddy current equations.
\newblock {\em J. Optim. Theory Appl.}, 154:879--903, 2012.

\bibitem{Yousept2012:2}
I.~Yousept.
\newblock Optimal control of {M}axwell's equations with regularized state
  constraints.
\newblock {\em Computational Optimization and Applications}, 52:559--581, 2012.

\bibitem{Yousept2013:1}
I.~Yousept.
\newblock Optimal bilinear control of eddy current equations with grad-div
  regularization and divergence penalization.
\newblock {\em Journal of Numerical Mathematics}, accepted for publication,
  2013.

\bibitem{Yousept2013:2}
I.~Yousept.
\newblock Optimal control of quasilinear
  $\boldsymbol{H}(\mathbf{curl})$-elliptic partial differential equations in
  magnetostatic field problems.
\newblock {\em SIAM Journal on Control and Optimization}, 51(5):3624--3651,
  2013.

\bibitem{ZoweKurcyusz1979}
J.~Zowe and S.~Kurcyusz.
\newblock Regularity and stability for the mathematical programming problem in
  {B}anach spaces.
\newblock {\em Applied Mathematics and Optimization}, 5(1):49--62, 1979.

\end{thebibliography}

\end{document}